\documentclass[reqno]{amsart}
\usepackage{amsmath}
\usepackage{amsfonts,amssymb, amstext}
\usepackage{mathrsfs}
\usepackage{soul}
\usepackage{mathtools}

\usepackage{url}
\usepackage{bbm}

\usepackage{color}

\usepackage{enumerate}

\makeatletter \@addtoreset{equation}{section} \makeatother

\renewcommand\thefigure{\thesection.\@arabic\c@figure}
\renewcommand\thetable{\thesection.\@arabic\c@table}
\newtheorem{theorem}{Theorem}[section]
\newtheorem{lemma}[theorem]{Lemma}

\theoremstyle{remark}

\newcommand{\mc}[1]{{\mathscr #1}}

\newcommand{\bb}[1]{{\mathbb #1}}

\newcommand{\CC}{\mathbb{C}}

\newcommand{\NN}{\mathbb{N}}

\newcommand{\RR}{\mathbb{R}}

\newcommand{\ZZ}{\mathbb{Z}}

\let\ve=\varepsilon

\let\ve=\varepsilon

\newcommand{\sfrac}[2]{{\smash{\frac{#1}{#2}}}}

\newcommand{\tnorm}{\vert \kern -1pt\vert\kern -1pt\vert}

\newcommand{\plus}{\!+\!}
\newcommand{\minus}{\!-\!}

\begin{document}
\title[Weakly Harmonic Chain] {Weakly harmonic oscillators perturbed by a conservative noise}

\author{C\'edric Bernardin}
\address{Universit\'e C\^ote d'Azur, CNRS, LJAD\\
Parc Valrose\\
06108 NICE Cedex 02, France}
\email{{\tt cbernard@unice.fr}}

\author{Patr\'{\i}cia Gon\c{c}alves}
\address{\noindent Center for Mathematical Analysis,  Geometry and Dynamical Systems \\
Instituto Superior T\'ecnico, Universidade de Lisboa\\
Av. Rovisco Pais, 1049-001 Lisboa, Portugal}
\email{patricia.goncalves@math.tecnico.ulisboa.pt}

\author{ Milton Jara}
\address{\noindent Instituto de Matem\'atica Pura e Aplicada\\ Estrada Dona Castorina 110\\ 22460-320 Rio De Janeiro, Brazil.}
\email{mjara@impa.br}

\begin{abstract}
We consider a chain of weakly harmonic coupled oscillators perturbed by a conservative noise. We show that by tuning accordingly the coupling constant energy can diffuse like a Brownian motion or superdiffuse like a maximally $3/2$-stable asymmetric L\'evy process. For a critical value of the coupling, the energy diffusion is described by a family of L\'evy processes which interpolates between these two processes.
\end{abstract}

\maketitle

\section{Introduction}

The problem of anomalous diffusion of energy in low dimensions has been the subject of intense research among last years, see \cite{LNP Lepri} for a recent review. In \cite{BBO2}, Hamiltonian chains of oscillators perturbed by conservative noise were introduced as a mathematically tractable model for energy superdiffusion. From the study of these models \cite{BGJ, JKO2, BBO2, BG14, KO}, the relevance of the different conservation laws in the origin of anomalous diffusion has started to be mathematically understood. These results have served as an inspiration and match perfectly the predictions of Spohn's fluctuating hydrodynamics theory \cite{S}, which allows to make precise conjectures on the decay of correlations of these chains. Fluctuating hydrodynamics predicts several universality classes for the behaviour of energy correlations in these chains, and a natural question is to understand how these different universality classes are related. In particular, we focus on the derivation of {\em crossovers} between different universality classes. We say that a family of equations, parametrized by some variable $\gamma$ is a {\em crossover} between two universality classes if the equations governing these universality classes are recovered taking the limits $\gamma \to 0$ and $\gamma \to \infty$. 

In \cite{BGJ}, \cite{JKO2}, it was proved that the scaling limit of the energy fluctuations of stochastic harmonic chains with at least two conserved quantities is governed by a fractional heat equation
\[
\partial_t u = {\mc L} u
\]
where ${\mc L}$ is the generator of a $3/2$-stable L\'evy process.  Notice that there is some freedom in the choice of ${\mc L}$, which corresponds to the {\em skewness} of the corresponding L\'evy process. The skewness of this operator actually depends on the sound velocity of the {\em mechanical} modes of the chain: volume in the case of \cite{BGJ} and momentum and stretch in the case of the model considered in \cite{JKO2}. These results correspond to the {\em zero-tension} universality class in the fluctuating hydrodynamics framework \cite{S,SS}.

When energy is the only conserved quantity in the chain, energy has normal diffusion, and energy fluctuations are governed by the usual heat equation, corresponding to the Edwards-Wilkinson universality class
\[
\partial_t u = \Delta u.
\]

In \cite{BGJSS,BGJS}, a crossover between these two universality classes was obtained by introducing a noise of vanishing intensity that destroys the conservation of volume. Tuning this intensity in a proper way, it was shown that energy fluctuations are governed by the evolution equation
\[
\partial_t u = \mc L_\gamma u,
\]
where $\mc L_\gamma$ is a non-local operator satisfying
\[
\lim_{\gamma \to 0} \mc L_\gamma = {\mc L},  \quad \lim_{\gamma \to \infty} \sqrt{\gamma} \mc L_\gamma = \Delta.
\]

In this article we take a different route, which we believe is more natural and provides a different interpolating operator. The stochastic chains introduced in \cite{BBO2,BS} are generated by an operator of the form $L_\kappa = \kappa A + S$, where  $A$ is the generator of the deterministic (Hamiltonian) part of the dynamics and $S$ is the generator of the stochastic part of the dynamics. The operator $A$ is antisymmetric with respect to the Gibbs measures associated to the chain, and $S$ is symmetric. If the operator $A$ is absent from the dynamics, namely $\kappa =0$, energy is diffusive. If $\kappa =1$, energy is superdiffusive \cite{BGJ}. Therefore, our task is to find how should $\kappa$ decay in order to observe an evolution different from the cases $\kappa=0$ and $\kappa=1$. 

This situation is very reminiscent of what happens in the so-called {\em weakly asymmetric}, one-dimensional exclusion process. The relevant quantity there is the current of particles through the origin. In the asymmetric case, it has been proved \cite{Joh} that the current of particles through the origin converges to the celebrated Tracy-Widom law. In the symmetric case, fluctuations are Gaussian \cite{RosVar,JarLan}. The crossover regime appears for $\kappa := \kappa_n= \frac{\gamma}{n^{1/2}}$, where $\tfrac{1}{n} \to 0$ is the space scaling, and it has been described in \cite{SaSp, ACQ}. In a more elaborated development, the {\em KPZ equation} serves as a crossover equation between the Edwards-Wilkinson universality class and the KPZ universality class. 

In our situation, since the macroscopic model is linear, fluctuations are always Gaussian, but the covariance structure changes drastically with $\kappa$. It turns out that the crossover scaling is $\kappa = \frac{\gamma}{n^{1/3}}$, and the crossover operator is given by
\[
\bb L_{\gamma,1/3} = \Delta + \sqrt{\gamma} {\mc L}. 
\]
It is clear that $\bb L_{0,1/3} = \Delta$ and
\[
\lim_{\gamma \to \infty} \tfrac{1}{\sqrt{\gamma}} \bb L_{\gamma,1/3} ={\mc L},
\]
which is the operator governing the energy fluctuations for $\kappa =1$.

Notice that the scaling $\kappa = \frac{\gamma}{n^{1/3}}$ differs from the crossover scaling of the weakly asymmetric exclusion process. In order to explain this scaling, we need to describe the results in \cite{BGJ} in more detail. Recall that the stochastic chains considered here have two conserved quantities: the {\em energy} and the so-called {\em volume}. It turns out that the volume serves as a fast variable for the evolution of the energy. Let us take the chain generated by the operator {$L_{\kappa,\lambda} = \kappa A  + \lambda S$}. By writing in a subdiffusive time-scale $t n^{3/2}$ the time evolution of the space-time energy-energy correlation $u_{t} (x)$ and of the space-time  energy-volume correlation $\psi_t(x,y)$, it can be shown that they are well approximated by the solution of 
\begin{equation}
\label{eq:approximative}
\begin{cases}
\; \partial_t u_t(x) = -2\lambda \partial_y^2 \psi_t(x,0)+ \tfrac{\lambda}{\sqrt n} \partial_x^2 u_t(x),\\
\; -\kappa \partial_x \psi_t(x,y) + \lambda \partial_y^2 \psi_t(x,y) = 0,\\
\; 4\lambda \partial_y \psi_t(x,0) = -\kappa \partial_x u_t(x).\\
\end{cases}
\end{equation}
Notice the presence of the correction term $\tfrac{\lambda}{\sqrt n} \partial_x^2 u$. This coupled system can be solved in $\psi$, giving an effective equation for $u_t$:
\[
\partial_t u_t(x) = \tfrac{\kappa^{3/2}}{\lambda^{1/2}} \mc L u_t(x) + \tfrac{\lambda}{\sqrt n} \partial_x^2 u_t(x).
\]
If we choose $\kappa = \lambda^{1/3}$, the first operator does not change with $\lambda$. We conclude that for $\lambda \ll \sqrt n$, the Laplacian part of this equation vanishes in the scaling limit, while for $\lambda = {c}{\sqrt n}$, $c>0$, the Laplacian has a finite, non-trivial contribution in the limit. One can check that $\lambda ={c}{\sqrt n}$ in the subdiffusive time scale $t n^{3/2}$ corresponds, in the diffusive time scale $tn^2$, to the choice $\kappa = \frac{\gamma}{n^{1/3}}$ mentioned above with $\gamma=c^{1/3}$ since $n^{3/2} \kappa= n^{3/2} \lambda^{1/3} = n^2 \, \tfrac{c^{1/3}}{ n^{1/3}}$. 

The fact that (\ref{eq:approximative}) is a good approximation is based on the approach initiated in \cite{BGJ}. However to control the error terms appearing in this approximation we have to be more clever than in \cite{BGJ}. Indeed, one of them can not be estimated by a {\textit{static}} estimate and is controlled by a {\textit{dynamical}} argument  (see the discussion after Lemma \ref{l3}). 

The paper is organized as follows. In Section \ref{sec:model-results} we describe the model and the main result. The proof is  given in Section \ref{sec:proofth} and the technical computations appear in the Appendix.

\section{The model}
\label{sec:model-results}
\subsection{Description of the model}

For $\eta: \bb Z \to \bb R$ and $\alpha >0$, define
\begin{equation}
\tnorm \eta \tnorm_{\alpha} = \sum_{x \in \bb Z} \big| \eta(x) \big| e^{-\alpha |x|}
\end{equation}
and let $\Omega_\alpha = \{ \eta: \bb Z \to \bb R; \tnorm \eta \tnorm_\alpha < +\infty\}$. The normed space $(\Omega_\alpha, \tnorm \cdot \tnorm)$ turns out to be a Banach space. In $\Omega_\alpha$ we consider the system of coupled ODE's
\begin{equation}
\label{ODE}
\tfrac{d}{dt} \tilde{\eta}_t(x) = \kappa \big[ {\tilde \eta}_t(x+1) - \tilde{\eta}_t(x-1)\big] \text{ for } t \geq 0 \text{ and } x \in \bb Z
\end{equation}
where $\kappa>0$ is a constant.The Picard-Lindel\"of Theorem shows that the system \eqref{ODE} is well posed in $\Omega_\alpha$. We will superpose to this deterministic dynamics a stochastic dynamics as follows. To each bond $\{x,x+1\}$, with $x \in \bb Z$ we associate an exponential clock of rate one. Those clocks are independent among them. Each time the clock associated to $\{x,x+1\}$ rings, we exchange the values of $\tilde \eta_t(x)$ and $\tilde \eta_t(x+1)$. Since there is an infinite number of such clocks, the existence of this dynamics needs to be justified. If we freeze the clocks associated to bonds not contained in $\{-M,\dots,M\}$, the dynamics is easy to define, since it corresponds to a piecewise deterministic Markov process. It can be shown that for an initial data $\eta_0$ in
\begin{equation}
\Omega = \bigcap_{\alpha>0} \Omega_\alpha,
\end{equation}
these piecewise deterministic processes stay at $\Omega$ and they converge to a well-defined Markov process $\{\eta_t; t\geq 0\}$, as $M \to \infty$, see \cite{BS} and references therein. This Markov process is the rigorous version of the dynamics described above. Notice that $\Omega$ is a complete metric space with respect to the distance
\begin{equation}
d(\eta,\xi) = \sum_{\ell \in \bb N} \frac{1}{2^\ell} \min\{ 1, \tnorm \eta-\xi\tnorm_{\frac{1}{\ell}}\}.
\end{equation}
Let us describe the generator of the process $\{\eta_t; t\geq 0\}$. For $x,y \in \bb Z$ and $\eta \in \Omega$ we define $\eta^{x,y} \in \Omega$ as
\begin{equation}
\eta^{x,y}(z)=
\begin{cases}
\eta(y); & z=x\\
\eta(x); &z=y\\
\eta(z); &z \neq x,y.
\end{cases}
\end{equation}
We say that a function $f: \Omega \to \bb R$ is \textit{ local } if there exists a finite set $B \subseteq \bb Z$ such that $f(\eta) = f(\xi)$ whenever $\eta(x) = \xi(x)$ for any $x \in B$. For a smooth function $f: \Omega \to \bb R$ we denote by $\partial_x f: \Omega \to \bb R$ its partial derivative with respect to $\eta(x)$. For a function $f: \Omega \to \bb R$ that is local, smooth and bounded, we define $L_\kappa  f: \Omega \to \bb R$ as $L_{\kappa} f = S f + \kappa A f$, where for any $\eta \in \Omega$,
\begin{equation}
S f(\eta)  = \sum_{x \in \bb Z} \big(f(\eta^{x,x+1}) - f(\eta)\big),
\end{equation}
\begin{equation}
Af(\eta) = \sum_{x \in \bb Z} \big( \eta (x+1) -\eta (x-1) \big)\, \partial_x f (\eta).
\end{equation}

The process $\{\eta_t; t\geq 0 \}$ has a family $\{\mu_{\rho,\beta}; \rho \in \bb R, \beta >0\}$ of invariant measures given by
\begin{equation}
\mu_{\rho,\beta}(d\eta) = \prod_{x \in \bb Z} \sqrt{\tfrac{\beta}{2 \pi}} \exp\big\{ - \tfrac{\beta}{2}\big(\eta(x) - \rho \big)^2 \big\} d\eta(x).
\end{equation}
It also has two conserved quantities. If one of the numbers
\begin{equation}
\sum_{x \in \bb Z} \eta_0(x) , \quad \sum_{x \in \bb Z} \eta_0(x)^2
\end{equation}
is finite, then its value is preserved by the evolution of $\{\eta_t; t \geq 0\}$. Following \cite{BS}, we will call these conserved quantities \textit{volume} and \textit{energy}, {respectively}. Notice that $\int \eta(x) d\mu_{\rho,\beta} = \rho$ and $\int\eta(x)^2 d\mu_{\rho,\beta} = \rho^2 + \frac{1}{\beta}$.

\subsection{Description of the result}
Fix $\rho \in \bb R$ and $\beta>0$, and consider the process $\{\eta_t; t \geq 0\}$ with initial distribution $\mu_{\rho,\beta}$. Notice that $\{\eta_t+\lambda; t \geq 0\}$ has the same distribution of the process {$\{\eta_t; t \geq 0\}$} with initial measure $\mu_{\rho+\lambda,\beta}$. Therefore, we can assume, without loss of generality, that $\rho = 0$. We will write $\mu_\beta= \mu_{0,\beta}$ and we will denote by $\bb P$ the law of $\{\eta_t; t \geq 0\}$ and by $\bb E$ the expectation with respect to $\bb P$. The \textit{ energy correlation function} $\{S_t(x); x \in \bb Z, t \geq 0\}$ is defined as
\begin{equation}
S_t(x) = \tfrac{\beta^2}{2} \,  \bb E\big[\big(\eta_0(0)^2-\tfrac{1}{\beta} \big) \big( \eta_t(x)^2 -\tfrac{1}{\beta} \big)\big]
\end{equation}
for any $x \in \bb Z$ and any $t \geq 0$. The constant $\frac{\beta^2}{2}$ is just the inverse of the variance of $\eta(x)^2-\frac{1}{\beta}$ under $\mu_\beta$. By translation invariance of the dynamics and the initial distribution $\mu_\beta$, we see that
\begin{equation}
\tfrac{\beta^2}{2}\,  \bb E\big[ \big(\eta_0(x)^2 -\tfrac{1}{\beta} \big) \big(\eta_t(y)^2 -\tfrac{1}{\beta}\big)\big]=S_t(y-x)
\end{equation}
for any $x, y \in \bb Z$. The following result was proved{\footnote{In \cite{BGJ}, $\kappa=1$ was assumed for simplicity but it is straightforward to extend the results there to cover the case $\kappa \ne 1$.}

\begin{theorem}[\cite{BGJ}]
\label{t1}
Assume $\kappa>0$. Let $f,g: \bb R \to \bb R$ be smooth functions of compact support. Then,
\begin{equation}
\label{ec1.120}
\lim_{n \to \infty} \tfrac{1}{n}\sum_{x, y  \in \bb Z} f\big( \tfrac{x}{n} \big) g\big( \tfrac{y}{n}\big) S_{tn^{3/2}}(x-y) = \iint f(x)g(y) P_t(x-y) dx dy,
\end{equation}
where $\{P_t(x); x \in \bb R, t \geq 0\}$ is the fundamental solution of the fractional heat equation
\begin{equation}
\label{ec1.1300}
\partial_t u =  -\sqrt{\cfrac{\kappa}{2}} \; \big\{ (-\Delta)^{3/4} - \nabla (-\Delta)^{1/4}\big\} u.
\end{equation}
\end{theorem}

It is not difficult to check that if $\kappa=0$ then the following result holds. 
\begin{theorem}
\label{t2}
Assume $\kappa=0$. Let $f,g: \bb R \to \bb R$ be smooth functions of compact support. Then,
\begin{equation}
\label{ec1.121}
\lim_{n \to \infty} \tfrac{1}{n}\sum_{x, y  \in \bb Z} f\big( \tfrac{x}{n} \big) g\big( \tfrac{y}{n}\big) S_{tn^{2}}(x-y) = \iint f(x)g(y) P_t(x-y) dx dy,
\end{equation}
where $\{P_t(x); x \in \bb R, t \geq 0\}$ is the fundamental solution of the heat equation
\begin{equation}
\label{ec1.1311}
\partial_t u = \Delta u.
\end{equation}
\end{theorem}

\textcolor{black}{We note that the previous results are obtained by looking at the system in different time scales: either in a superdiffusive time scale $tn^{3/2}$ or in the diffusive time scale $tn^2$.}
Our aim is now to investigate a crossover between these two regimes by letting $\kappa$ to be arbitrary small. To this end, we introduce a large parameter $n \in \NN$ and take 
\begin{equation}
\textcolor{black}{\kappa:=\kappa_n=\frac{\gamma}{n^b}}
\end{equation}
where $\gamma>0$ is fixed and $b\in (0,+\infty)$. Observe that the two previous theorems describe, respectively, the limiting cases $b=0$ and $b=+\infty$. The main result of this paper is the following theorem.

\begin{theorem}
\label{t3}
Assume $\kappa_n= \tfrac{\gamma}{n^b}$ where $b\ge 0$ and $\gamma>0$. We define the exponent $a$ of the time scale by $a=\inf ( 3/2- 3b/2, 2)$. Let $f,g: \bb R \to \bb R$ be smooth functions of compact support. Then,
\begin{equation}
\label{ec1.122}
\lim_{n \to \infty} \tfrac{1}{n}\sum_{x, y  \in \bb Z} f\big( \tfrac{x}{n} \big) g\big( \tfrac{y}{n}\big) S_{tn^{a}}(x-y) = \iint f(x)g(y) P^{\gamma,b}_t (x-y) dx dy,
\end{equation}
where $\{P^{\gamma,b}_t(x); x \in \bb R, t \geq 0\}$ is the fundamental solution of the equation
\begin{equation}
\label{ec1.1320}
\partial_t u = {\bb L}_{\gamma,b} u
\end{equation}
with ${\bb L}_{\gamma,b}$ \textcolor{black}{being} the generator of the following \textcolor{black}{L\' evy} process
\begin{equation}
\label{eq:Levy356}
{\bb L}_{\gamma, b} = {\bf 1}_{b \ge 1/3}\; \Delta + \sqrt{\gamma} \, {\bf 1}_{b \le 1/3}  \left[ - \tfrac{1}{\sqrt 2} \; \big\{ (-\Delta)^{3/4} - \nabla (-\Delta)^{1/4}\big\} \right].
\end{equation}
\end{theorem}

The most interesting regime is for $b=1/3$ since the operator ${\bb L}_{\gamma, 1/3}$ is a L\'evy process connecting the Brownian motion to the totally asymmetric $3/2$-stable L\'evy process: 
$${\bb L}_{\gamma,1/3} \underset{\gamma \to 0} {\longrightarrow} \Delta$$
and 
$$\gamma^{-1/2}\, {\bb L}_{\gamma,1/3} \underset{\gamma \to \infty} {\longrightarrow} \left[ - \tfrac{1}{\sqrt 2} \; \big\{ (-\Delta)^{3/4} - \nabla (-\Delta)^{1/4}\big\} \right].$$

\section{Proof of Theorem \ref{t3}}
\label{sec:proofth}

Following the method introduced in \cite{BGJ}, the proof of this theorem will be established by a careful study of the correlation function $\{S_t(x,y); x \neq y \in \bb Z, t \geq 0\}$ given by
\begin{equation}
S_t(x,y) = \tfrac{\beta^2}{2} \bb E \big[ \big(\eta_0(0)^2-\tfrac{1}{\beta} \big) \eta_t(x) \eta_t(y)\big]
\end{equation}
for any $t \geq 0$ and any $x \neq y \in \bb Z$. Notice that this definition makes perfect sense for $x =y$ and, in fact, we have $S_t(x,x) = S_t(x)$. For notational convenience we define $S_t(x,x)$ as equal to $S_t(x)$. However, these quantities are of different nature, since $S_t(x)$ is related to \textit{energy fluctuations} and $S_t(x,y)$ is related to {\textit{volume fluctuations} (for $x \neq y$).

Let $a=\inf( 3/2- 3b/2, 2)$ which fixes the time scale in which we observe the process at. The generator $n^{a} L_{\kappa_n}$ is denoted by ${\mc L}_n$. From now on we assume $\beta=1$, since the general case can be recovered from this particular case by multiplying the process by $\beta^{-1/2}$. 

For $d \ge 1$, denote by $\mc C_c^\infty(\bb R^d)$ the space of infinitely differentiable functions $f:{\bb R}^d \to \bb R$ of compact support. For any function $f \in \mc C_c^\infty(\bb R^d)$, define the \textcolor{black}{discrete}  $\ell^2({\bb Z}^d)$-norm as
\begin{equation}
\|f\|_{2,n} = \sqrt{ \vphantom{H^H_H}\smash{\tfrac{1}{n^d} \sum_{x \in {\bb Z}^d} f\big( \tfrac{x}{n} \big)^2}}.
\end{equation}
\vspace{0pt}

 Let $g \in \mc C_c^\infty(\bb R)$ be a fixed function.
For any $n \in \bb N$,  \textcolor{black}{any}  $t \geq 0$ and any $f \in \mc C_c^\infty(\bb R)$, we define the field $\{\mc S_t^n; t \geq 0\}$ as
\begin{equation}
\mc S_t^n(f) = \tfrac{1}{n}\!\! \sum_{x, y \in \bb Z} g\big(\tfrac{\vphantom{y}x}{n} \big) f\big( \tfrac{y}{n}\big) S_{tn^{a}}(y-x).
\end{equation}
\textcolor{black}{We observe that the previous field is the one which appears at the left hand side of \eqref{ec1.122}.} 
Rearranging terms in a convenient way we have that
\begin{equation}
\mc S_t^n(f) = \tfrac{1}{2} \bb E \Big[ \Big(\tfrac{1}{\sqrt n} \sum_{x \in \bb Z} g\big(\tfrac{x}{n}\big)\big(\eta_0(x)^2-1\big) \Big) \times
	\Big( \tfrac{1}{\sqrt n} \sum_{y \in \bb Z}f\big(\tfrac{y}{n}\big)  \big(\eta_{tn^{3/2}}(y)^2 -1\big)  \Big)\Big].
\end{equation}
By a simple application of the Cauchy-Schwarz inequality we have that
\begin{equation}\label{apriori_bound_S}
|\mc S^n_t(f)|\leq \|g\|_{2,n}\|f\|_{2,n}.
\end{equation}
For a function $h \in \mc C_c^\infty(\bb R^2)$ we define $\{ Q_t^n(h); t \geq 0\}$ as
\begin{equation}
 Q_t^n(h) = \tfrac{1}{2} \bb E \Big[ \Big(\tfrac{1}{\sqrt n} \sum_{x \in \bb Z} g\big(\tfrac{x}{n}\big)\big(\eta_0(x)^2-1\big) \Big) \times
	\Big( \tfrac{1}{ n}\!\! \sum_{y \neq z  \in \bb Z}\!\! h\big(\tfrac{y}{n}, \tfrac{\vphantom{y}z}{n}\big) \eta_{tn^{a}}(y) \eta_{tn^{a}}(z)   \Big)\Big].
\end{equation}
Note that $Q_t^n(h)$ depends only on the symmetric part of the function $h$ and, along this article, we will always assume, without loss of generality, that $h(x,y) = h(y,x)$ for any $x,y \in \bb Z$. We also point out that $Q_t^n(h)$ does not depend on the values of $h$ at the diagonal $\{x=y\}$. Again, by a simple application of the Cauchy-Schwarz inequality we have that
\begin{equation}\label{apriori_bound_Q}
|Q^n_t(h)|\leq 2 \|g\|_{2,n}\|h\|_{2,n}.
\end{equation}

Note that, by a simple computation, whose details are given in Appendix \ref{sec:Abb}, we have that
\begin{equation}
 \label{ec3.10}
\tfrac{d}{dt} \mc S_t^n(f) = -2\gamma n^{a-b -3/2} \, Q_t^n(\nabla_n f \otimes \delta) +  n^{a-2} \, \mc S_t^n(\Delta_n f)
\end{equation}
where for a function $f \in \mc C_c^\infty(\bb R)$,  $\Delta_n f: \bb R \to \bb R$ is a discrete approximation of the second derivative of $f$ given by
\begin{equation}
\Delta_n f \big( \tfrac{x}{n}\big) = n^2\Big( f\big(\tfrac{x\plus 1}{n} \big) + f\big( \tfrac{x\minus 1}{n} \big) - 2 f\big(\tfrac{x}{n}\big) \Big)
\end{equation}
and  $\nabla_n f \otimes \delta : \sfrac{1}{n} \bb Z^2 \to \bb R$ is a discrete approximation of the distribution $f'\!(x) \otimes \delta(x=y)$, where $\delta(x=y)$ is the $\delta$ of Dirac at the line $x=y$ and it is given by
\begin{equation}
\label{eq:3.9}
\big(\nabla_n f \otimes \delta\big) \big( \tfrac{x}{n}, \tfrac{y}{n}\big) =
\begin{cases}
\frac{n^2}{2}\big(f\big(\tfrac{x+1}{n}\big) - f\big(\tfrac{x}{n}\big)\big); & y =x\plus 1\\
\frac{n^2}{2}\big(f\big(\tfrac{x}{n}\big) - f\big(\tfrac{x-1}{n}\big)\big); & y =x\minus 1\\
0; & \text{ otherwise.}
\end{cases}
\end{equation}

At this point we \textcolor{black}{ split the proof of the theorem according to the range of the parameter $b$. First we treat the case $b>1$. For that purpose we  note that putting together (\ref{apriori_bound_Q}) plus  the fact that $\| \nabla_n f \otimes \delta\|_{2,n}=\mc O(\sqrt{n})$,  we get that for  $b>1$ and $a=2$, 
$$\lim_{n \to \infty}  n^{a-b -3/2} \, Q_t^n(\nabla_n f \otimes \delta) =0.$$
This concludes the proof of the theorem for $b>1$, since the equation \eqref{ec3.10} for $\mc S_t^n(f)$ is now closed and a simple tightness argument gives the result. Now we note that applying the  $H_{-1}$-norm argument as in the proof of Theorem 4 in \cite{BG14} we get that for $b>1/2$, the previous result still holds. We do not present this proof  here since it is a reproduction of the arguments of Theorem 4 of  \cite{BG14} but we ask the interested reader to look particularly  at the last inequality of the proof of that theorem  and to note that the time scaling for this range of $b$ is $a=2$. At this point we still need to analyze the remaining cases where  $b\leq 1/2$. 
We also point out that 
the previous arguments do not use the asymmetric part of the dynamics in order to control the problematic term $ n^{a-b -3/2} \, Q_t^n(\nabla_n f \otimes \delta)$. The understanding of the effect of the asymmetric part is crucial to cover the case $b\leq 1/2$. From now on, we assume  that this is the case (in fact the rest of the argument is valid for $b<1$ but not for $b=1$). } Since, from \eqref{ec3.10} the time evolution of $\mc S^n_t$ depends on the time evolution of $Q^n_t$, we need also to find an equation similar to \eqref{ec3.10} for $Q^n_t$, in order to close the equation for $\mc S_t^n.$ \textcolor{black}{ We note that this argument has already been used in \cite{BGJ} when treating the case corresponding to  $b=0$. } By a simple computation, whose details are given in Appendix \ref{sec:Abb}, we have that
\begin{align}
\label{ec3.15}
\tfrac{d}{dt} Q_t^n(h) &= Q_t^n\big(n^{a-2}\Delta_n h+ \gamma n^{a-b-1}\mc A_n h\big) -2 \gamma n^{a-b-3/2} \mc S_t^n\big( \mc D_n h\big) \\
&+ 2 Q_t^n\big( n^{a-2} \widetilde{\mc D}_n h \big), 
\end{align}
where for $h \in \mc C_c^\infty(\bb R^2)$ the operator  $\Delta_n h : \bb R^2 \to \bb R$ is a  discrete approximation of the \textcolor{black}{$2$}-dimensional Laplacian of $h$  and it is given by
\begin{equation}
\Delta_n h\big( \tfrac{\vphantom{y}x}{n}, \tfrac{y}{n}\big) = n^2\Big( h\big( \tfrac{\vphantom{y}x+1}{n}, \tfrac{y}{n}\big)+h\big( \tfrac{\vphantom{y}x-1}{n}, \tfrac{y}{n}\big)+h\big( \tfrac{\vphantom{y}x}{n}, \tfrac{y+1}{n}\big)+ h\big( \tfrac{\vphantom{y}x}{n}, \tfrac{y-1}{n}\big) - 4 h\big( \tfrac{\vphantom{y}x}{n}, \tfrac{y}{n}\big)\Big),
\end{equation}
 $\mc A_n h: \bb R \to \bb R$ is a  discrete approximation of the directional derivative $(-2,-2) \cdot \nabla h$ and is given by
\begin{equation}
\mc A_n h\big( \tfrac{\vphantom{y}x}{n}, \tfrac{y}{n}\big) = n\Big(h\big( \tfrac{\vphantom{y}x}{n}, \tfrac{y-1}{n}\big)+ h\big( \tfrac{\vphantom{y}x-1}{n}, \tfrac{y}{n}\big)- h\big( \tfrac{\vphantom{y}x}{n}, \tfrac{y+1}{n}\big)-h\big( \tfrac{\vphantom{y}x+1}{n}, \tfrac{y}{n}\big)\Big),
\end{equation}
 the operator $\mc D_n h : \sfrac{1}{n} \bb Z \to \bb R$ is a discrete approximation of the directional derivative of $h$ along the diagonal $x=y$ and it is  given by
\begin{equation}
\mc D_n h\big( \tfrac{x}{n} \big) = n \Big( h \big(\tfrac{x}{n}, \tfrac{x+1}{n}\big) - h \big( \tfrac{x-1}{n}, \tfrac{x}{n} \big) \Big)
\end{equation}
and the operator $\widetilde {\mc D}_n$ is defined as follows: $\widetilde {\mc D}_n h$ is a symmetric function from $\tfrac{1}{n} \ZZ^2$ into $\RR$, vanishing outside the upper and lower diagonal, and such that
\begin{equation}
\label{eq:D_tilde}
\widetilde{\mc D}_n h  \; \big(\tfrac{x}{n},\tfrac{x+1}{n}\big)   =n^2 \left[ \widetilde{\mc E}_n h \, \big(\tfrac{x}{n}\big)   -\tfrac{1-\kappa}{2} \; \widetilde{\mc F}_n h\,   \big(\tfrac{x}{n}\big) \right] 
\end{equation}
with
\begin{equation}\label{eq:En}
\widetilde{\mc E}_n h \big(\tfrac{x}{n}\big) = 
 h \big( \tfrac{x}{n} , \tfrac{x+1}{n} \big) -h \big( \tfrac{x}{n} , \tfrac{x}{n} \big)\end{equation}
and
\begin{equation}\label{eq:Fn}
\widetilde{\mc F}_n h (\tfrac{x}{n},\tfrac{y}{n}) =  h \big( \tfrac{x+1}{n} , \tfrac{x+1}{n} \big) -h \big( \tfrac{x}{n} , \tfrac{x}{n} \big).
\end{equation}

In order to combine \eqref{ec3.10} and \eqref{ec3.15} in such a way that we obtain a simple equation for the time evolution of $\mc S_t^n$ we consider  $h_n: \sfrac{1}{n} \bb Z^2 \to \bb R$ as the solution of the \textcolor{black}{Poisson} equation
\begin{equation}
\label{Poisson}
\Delta_n h \big(\tfrac{\vphantom{y}x}{n},\tfrac{y}{n}\big) + \gamma n^{1-b} \mc A_n h \big(\tfrac{\vphantom{y}x}{n},\tfrac{y}{n}\big)  =  2\gamma n^{1/2-b} \, \nabla_n f \otimes \delta \big(\tfrac{\vphantom{y}x}{n},\tfrac{y}{n}\big).
\end{equation}
Note that $h_n$ is independent of $a$. We get that
\begin{equation}
\tfrac{d}{dt} \mc S_t^n(f) = - \tfrac{d}{dt} Q_t^n(h_n) + {\mc S}_t^n \Big( n^{a-2} \Delta_n f -2 \gamma n^{a-b-3/2} {\mc D}_n h_n  \Big) + 2 Q_t^n\big( n^{a-2} \widetilde{\mc D}_n h_n \big).
\end{equation}
By integrating last expression in time we have that for $T>0$,
\begin{equation}
\label{dec_S}
\begin{split}
 \mc S_T^n(f)-  \mc S_0^n(f)=& ~\int_0^T{\mc S}_t^n \Big( n^{a-2} \Delta_n f -2 \gamma n^{a-b-3/2} {\mc D}_n h_n  \Big)\, dt \\&+Q_0^n(h_n) - Q_T^n(h_n) + 2\int_0^T Q_t^n\big( n^{a-2} \widetilde{\mc D}_n h_n \big)\, dt.
\end{split}\end{equation}
Now we need to analyze each term at the right hand side of (\ref{dec_S}). Let us first observe that for any $b \ge 0$
\begin{equation}
\label{eq:l31}
\lim_{n \to \infty}\|h_n\|_{2,n}^2=0.
\end{equation}
This is proved in \textcolor{black}{Appendix \ref{sec:proof_l3_1}}. We note then that by \eqref{apriori_bound_Q} the second and third terms at the right hand side of (\ref{dec_S}) vanish, as $n\to\infty$. 
The contribution of the main term at the right hand side of (\ref{dec_S}) is encapsulated in the following lemma which is proved in \textcolor{black}{ Appendix \ref{sec:proof_l3_2}}.
\begin{lemma}
\label{l3}
Let $f \in \mc C_c^\infty(\bb R)$. If $a=\inf(3/2+3b/2, 2)$ and $b\in(0,+\infty)$, \textcolor{black}{then}
\begin{equation}
\label{eq:l32}
\lim_{n \to \infty} \frac{1}{n} \sum_{x \in \bb Z} \big| \big\{n^{a-2} \Delta_n f -2 \gamma n^{a-b-3/2} {\mc D}_n h_n\big\}\big(\tfrac{x}{n}\big)  -{\bb L}_{\gamma,b} f\big(\tfrac{x}{n}\big) \big|^2 =0
\end{equation}
where ${\bb L}_{\gamma,b}$ is defined in (\ref{eq:Levy356}).
\end{lemma}

\textcolor{black}{The last term at the \textcolor{black}{ right hand side}  of (\ref{dec_S}) is not so easy to control combining the  Cauchy-Schwarz estimate with a bound on the $\bb L^2$-norm of the function involved, that is of $n^{a-2}\widetilde{ \mc D}_n h_n$. Nevertheless, by repeating the same argument as above, that is, by rewriting  the time evolution of the field $ Q_t^n$ in terms of a solution of another Poisson equation which gives us an expression for the term at the right hand side of equation (\ref{dec_S}), we obtain the following result:} 
\begin{lemma}
\label{lem:pasteque}
Let $h_n:\frac{1}{n}\bb Z^2\to\bb R$ be the solution of the Poisson equation given in \eqref{Poisson},  $a=\inf(3/2+3b/2, 2)$ and $b\in(0,1)$. For any $T>0$ we have that
 \begin{equation*}
\lim_{n\to\infty} \bb E\Big[\Big(\int_0^T Q_t^n (n^{a-2}\widetilde{\mc D}_n h_n)\,dt\Big)^2 \Big]=0.
 \end{equation*}
\end{lemma}

\begin{proof}
\textcolor{black}{As mentioned above, in order to prove the result we } use again (\ref{ec3.15}) applied to $h=v_n$ where $v_n$, is the solution of the Poisson equation
\begin{equation}
\label{eq:poisson-v}
n^{a-2} \Delta_n v_n \big(\tfrac{\vphantom{y}x}{n},\tfrac{y}{n}\big) + \gamma n^{a-1-b} \mc A_n v_n\big(\tfrac{\vphantom{y}x}{n},\tfrac{y}{n}\big)  =n^{a-2} {\widetilde {\mc D}}_n h_n.
\end{equation}

Then, by integrating in time \eqref{ec3.15} we have
\begin{equation*}
\begin{split}
\int_0^T Q_t^n \big( n^{a-2} {\widetilde{\mc D}}_n h_n \big) dt &= 2\gamma n^{a-b-3/2}  \int_0^T {\mc S}_t^n ({\mc D}_n v_n) dt -2 \int_{0}^T {Q}_t^n  \big( n^{a-2} {\widetilde{\mc D}}_n v_n \big) dt \\
&+ Q_T^n (v_n) -Q_0^n (v_n).
\end{split}
\end{equation*}

Now, by using repeatedly the inequality $(x+y)^2\leq 2x^2+2y^2$ in order to conclude we have to show that 
\textcolor{black}{
 \begin{equation}
 \label{new_da12}
\lim_{n\to\infty} \bb E\Big[\Big(n^{a-b-3/2}  \int_0^T {\mc S}_t^n ({\mc D}_n v_n)\,dt\Big)^2 \Big]=0, 
 \end{equation} 
\begin{equation}
\label{new_da13}
\lim_{n\to\infty} \bb E\Big[\Big( Q_T^n (v_n) -Q_0^n (v_n)\Big)^2 \Big]=0
\end{equation}
and
\begin{equation}\label{new_da}
\lim_{n\to\infty} \bb E\Big[\Big(\int_{0}^T {Q}_t^n  \big( n^{a-2} {\widetilde{\mc D}}_n v_n \big)\, dt \Big)^2 \Big]=0.
 \end{equation}}
We have the following estimates on $v_n$ which are proved in Appendix \ref{subsec:secondterm}.
\begin{lemma}
\label{lem:5,7.}
The solution $v_n$ of \eqref{eq:poisson-v} satisfies
\begin{equation}
\label{est vn}
\lim_{n \to \infty}\|v_n\|_{2,n}^2=0,
\end{equation}
\begin{equation}\label{est der vn}
\lim_{n \to \infty}\|n^{a-b -3/2} \,  \mc D_nv_n\|^2_{2,n}=0,
\end{equation}
\end{lemma}
The expectation in (\ref{new_da12}), by \eqref{apriori_bound_S} and \eqref{est der vn}, vanishes, as $n\to\infty$. Similarly the expectation in (\ref{new_da13}), by \eqref{apriori_bound_Q} and \eqref{est vn}, vanishes as $n\to\infty$.
\textcolor{black}{To bound  the third expectation above we could be tempted to combine the a priori bound \eqref{apriori_bound_Q} with an estimate on the $\bb L^2$ norm of $n^{a-2} {\widetilde{\mc D}}_n v_n $. We leave the interested reader to check that this argument  only shows \eqref{new_da} for $b>1$ (and $a=2$) or for $a<2$ (i.e. $b<1/3$). Therefore, proving \eqref{new_da} for $b>1/3$ requires extra work. To overcome this problem,  our  idea is  to establish the result of Lemma \ref{lem:pasteque} with $h_n$ replaced by $v_n$ by using the same method used in the current lemma but with the advantage that now the a priori bound \eqref{apriori_bound_Q} combined with the estimates on the $\bb L^2$-norms of the functions involved, will be sufficient to conclude  the proof  for any $b<1$.  This is the content of Lemma \ref{lem:new_da}, from where we conclude the proof of the theorem.}
\end{proof}

In fact, in order to complete the proof, some tightness of the fields have to be established. The arguments being similar to the ones given in \cite{BGJ}, Section 5.2, we do not repeat them and invite the interested reader to read the proofs there.

\section*{Acknowledgements}
This work was  supported by the French Ministry of Education through the grant ANR (EDNHS). C.B. thanks the French Ministry of Education for its support through the grant ANR (LSD). 

P.G.~thanks  FCT/Portugal for support through the project 
UID/MAT/04459/2013.  

 M.J.~thanks CNPq for its support through the grant 401628/2012-4 and FAPERJ for its support through the grant JCNE E17/2012. M.J.~was partially supported by NWO Gravitation Grant 024.002.003-NETWORKS. 
 
This research was supported in part by the International Centre for Theoretical Sciences (ICTS) during a visit for participating in the program Non-equilibrium statistical physics (Code: ICTS/Prog-NESP/2015/10).

\appendix

\section{Computations involving the generator $L_\kappa$}
\label{sec:Abb}

Let $f: \bb Z \to \bb R$ be a function of finite support, and let ${\mc E} (f) : \Omega \to \RR$ be defined as
\begin{equation*}
{\mc E} (f) =\sum_{{x \in \bb Z}} f(x)\eta(x)^2.
\end{equation*}
A simple computation shows that
\begin{equation*}
S{\mc E} (f) =\sum_{{x \in \bb Z}} \Delta f(x)\eta^2(x),
\end{equation*}
where $\Delta f(x)  =  f(x\!+\!1)+f(x\!-\!1) -2f(x)$ is the discrete Laplacian on $\bb Z$.
On the other hand
\begin{equation*}
A{\mc E} (f) =-2\sum_{{x \in \bb Z} } \nabla f(x)\eta(x)\eta(x\!+\!1),
\end{equation*}
where $\nabla f(x)  =  f(x\!+\!1) -f(x)$ is the discrete right-derivative on $\bb Z$. It follows that
\begin{equation}
L_{\kappa} {\mc E} (f) =-2\kappa \sum_{{x \in \bb Z} } \nabla f(x)\eta(x)\eta(x\!+\!1) + \sum_{{x \in \bb Z}} \Delta f(x)\eta^2(x).
\end{equation}

Let $f: \bb Z^2 \to \bb R$ be a symmetric function of finite support, and let ${ Q} (f) : \Omega \to \RR$ be defined as
\vspace{-10pt}
\begin{equation*}
{Q} (f) =\sum_{\substack{x,y \in \bb Z \\ x \neq y}} \eta (x) \eta (y) f(x,y).
\end{equation*}
Define $\Delta f: \bb Z^2 \to \bb R$ as
\begin{equation}
\Delta f(x,y)  =  f(x\!+\!1,y)+f(x\!-\!1,y) + f(x,y\!+\!1)+f(x,y\!-\!1)-4f(x,y)
\end{equation}
for any $x,y \in \bb Z$ and $\mc A f: \bb Z^2 \to \bb R$ by
\begin{equation}
\mc Af(x,y) = f(x\!-\!1,y) + f(x,y\!-\!1) - f(x\!+\!1,y) - f(x,y\!+\!1)
\end{equation}
for any $x,y \in \bb Z$. Notice that $\Delta f$ is the discrete Laplacian on the lattice $\bb Z^2$ and $\mc A f$ is a possible definition of the discrete derivative of $f$ in the direction $(-2,-2)$. Notice that we are using the same symbol $\Delta$ for the one-dimensional and two-dimensional, discrete Laplacian. From the context it will be clear which operator we will be using. We have that
\begin{equation}
\begin{split}
S Q (f)
		&=\!\! \sum_{|x-y| \ge 2}\!\! f(x,y) \big[ \eta (y) \Delta \eta (x) + \eta (x) \Delta \eta (y) \big]\\
		&+ 2 \sum_{x\in\bb Z} f(x,x\!+\!1) \big[ (\eta (x\!-\!1) -\eta (x)) \eta (x\!+\!1) +(\eta (x\!+\!2) -\eta (x\!+\!1)) \eta (x)\big]\\
		&=\!\! \sum_{x,y\in\bb Z}\!\! \Delta f(x,y) \eta (x) \eta (y) -2 \sum_{x \in \bb Z} f(x,x) \eta (x) \Delta \eta(x) \\
		&-2 \sum_{x\in\bb Z}  f(x,x\!+\!1) \big[ \eta (x\!+\!1) \Delta \eta(x) + \eta (x) \Delta \eta(x\!+\!1)\big]\\
		&+ 2 \sum_{x\in\bb Z}  f(x,x\!+\!1) \big[ \eta (x\!+\!1) \eta (x\!-\!1) +\eta (x\!+\!2) \eta (x) -2 \eta (x) \eta (x\!+\!1)\big].\\
\end{split}
\end{equation}
Grouping terms involving $\eta(x)^2$ and $\eta(x)\eta(x\!+\!1)$ together we get that
\begin{equation}
\begin{split}
S Q(f)
		&= \sum_{\substack{x,y \in \bb Z \\ x \neq y}} ({\Delta} f)(x,y) \eta (x)\eta (y) \\
		& + 2 \sum_{x \in \bb Z}  \Big\{ \big[f(x,x\!+\!1)-f(x,x)\big] +\big[f(x,x\!+\!1) -f(x\!+\!1, x\!+\!1)\big]\Big\} \eta (x)\eta(x\!+\!1)\\
		&= Q(\Delta f)\\ 
		&+ 2 \sum_{x\in\bb Z}  \Big\{\big[ f(x,x\!+\!1) -f(x,x)\big] +\big[f(x,x\!+\!1) -f(x\!+\!1, x\!+\!1)\big]\Big\} \eta (x)\eta(x\!+\!1).\\
\end{split}
\end{equation}
Similarly, we have that
\begin{equation}
\begin{split}
A{Q} (f) &= \!\!\sum_{\substack{x,y \in \bb Z \\ x \neq y}}\!\!  \mc A f(x,y) \eta(x) \eta(y)\\
		&+ 2 \sum_{x \in \bb Z} \Big\{ \eta(x)^2 \big[ f(x\!-\!1,x) -f(x,x\!+\!1)\big]  - \eta(x) \eta(x\!+\!1) \big[ f(x,x)-f(x\!+\!1,x\!+\!1)\big]\Big\}\\
&= Q(\mc A f)\\
		&+ 2 \sum_{x \in \bb Z} \Big\{ \eta(x)^2 \big[ f(x\!-\!1,x) -f(x,x\!+\!1)\big] - \eta(x) \eta(x\!+\!1) \big[ f(x,x)-f(x\!+\!1,x\!+\!1)\big]\Big\}.
\end{split}
\end{equation}
From this it follows that
\vspace{-10pt}
\begin{equation}
\label{ecA.8}
L_{\kappa} Q(f) = Q((\Delta + \kappa \mc A)f) + D_\kappa (f),
\end{equation}
where the diagonal term $D_\kappa(f)$ is given by
\begin{equation}
\begin{split}
D_\kappa (f)  &=2 \kappa \sum_{x \in \bb Z} \big(\eta(x)^2- \tfrac{1}{\beta}\big) \big(f(x\minus 1,x) - f(x,x \plus 1)\big)\\
	&\quad+ 2 \sum_{x \in \bb Z} \eta(x) \eta(x \plus 1) \big( 2 f(x,x \plus 1) -(1+\kappa) f(x,x) -(1-\kappa) f(x\plus 1, x\plus 1) \big).
\end{split}
\end{equation}
\textcolor{black}{Above in $D_\kappa$, we could add the normalization constant  $\frac{1}{\beta}$ for free, since  $\sum_{x\in\bb Z}f(x,x\plus 1) - f(x\minus 1,x)=0$. We also note that the operators $f \mapsto Q(f)$, $f \mapsto L Q(f)$ are continuous maps from $\ell^2(\bb Z^2)$ to $L^2(\mu_\beta)$ and therefore, an approximation procedure shows that the identities above are true for any $f \in \ell^2(\bb Z^2)$.}

\section{Tools of Fourier analysis}\label{sec:Ab}

Let $d \ge 1$ and let $x \cdot y $ denote the usual scalar product in $\RR^d$ between $x$ and $y$. The Fourier transform  of a function $g :\tfrac{1}{n} \ZZ^d \to \RR$ is defined by
\begin{equation*}
{\widehat g}_n (k) = \tfrac{1}{n^d} \sum_{x \in \ZZ^d} g (\tfrac{x}{n}) e^{ \tfrac{2i \pi k \cdot x}{n}}, \quad k\in \RR^d.
\end{equation*}
The function $\widehat{g}_n$ is $n$-periodic in all the directions of $\RR^d$. We have the following Parseval-Plancherel identity  between the $\ell^2$-norm of $g$, weighted by the natural mesh, and the $L^2 ([-\tfrac{n}{2}, \tfrac{n}{2}]^d)$-norm of its Fourier transform:
\begin{equation}
\| g \|^2_{2,n} := \tfrac{1}{n^d} \sum_{x \in \ZZ^d} | g (\tfrac{x}{n})|^2 = \int_{[-\tfrac{n}{2}, \tfrac{n}{2}]^d}\;  \left| {\widehat g}_n (k) \right|^2 \, dk \; := \| {\widehat g_n} \|_2^2.
\end{equation}

The function $g$ can be recovered from the knowledge of its Fourier transform by the inverse Fourier transform of ${\widehat g}_n$:
\begin{equation}
g ( \tfrac{x}{n}) = \int_{[-\tfrac{n}{2}, \tfrac{n}{2}]^d}\; {\widehat g}_n (k)\;  e^{- \frac{2i \pi x\cdot k}{n}} \, dk.
\end{equation}

For any $p\ge 1$ let $[(\nabla_n)^p ]$ denote the $p$-th iteration of the operator $\nabla_n$.

\begin{lemma}[\cite{BGJ}]
\label{lem:sfp}
Let $f: \tfrac{1}{n} \ZZ \to \RR$ and $p \ge 1$ be such that
\begin{equation}
\label{eq:asssfp}
\frac{1}{n} \sum_{x \in \ZZ} \left| [(\nabla_n)^p ] f\big( \tfrac{x}{n} \big) \right| < +\infty.
\end{equation}
There exists a universal constant $C:=C(p)$ independent of $f$ and $n$  such that for any $|y| \le 1/2$,
\begin{equation*}
| \widehat{f_n} (yn) | \le \frac{C}{n^p |\sin(\pi y)|^p} \; \left| \frac{1}{n} \sum_{x \in \ZZ} [(\nabla_n)^p ] f\big( \tfrac{x}{n} \big) e^{2i \pi y x} \right|.
\end{equation*}
In particular, if $f$ is in the Schwartz space  ${\mc S} (\RR)$, then for any $p\ge 1$, there exists a constant $C:=C(p,f)$ such that for any $|y| \le 1/2$,
\begin{equation*}
| \widehat{f_n} (yn) | \le \frac{C}{1+ (n|y|)^{p}}.
\end{equation*}
\end{lemma}

Several times we will use the following elementary change of variable property.

\begin{lemma}[\cite{BGJ}]
\label{lem:cov}
Let $F: \RR^2 \to \CC$ be a $n$-periodic function in each direction of $\RR^2$. Then we have that
\begin{equation*}
\iint_{[-\tfrac{n}{2}, \tfrac{n}{2}]^2} F (k,\ell) \, dk d\ell \; = \; \iint_{[-\tfrac{n}{2}, \tfrac{n}{2}]^2} F(\xi- \ell, \ell) \, d\xi d\ell.
\end{equation*}
\end{lemma}

\section{Estimates involving $h_n$}
\label{sec:proof_l3}
\textcolor{black}{Let $h_n: \sfrac{1}{n} \bb Z^2 \to \bb R$ be the unique solution in $\ell^2(\sfrac{1}{n}\bb Z^2)$ of \eqref{Poisson}. Observe that $h_n$ is a symmetric function. The Fourier transform of $h_n$ is not difficult to compute  by using Appendix \ref{sec:Ab}.
First we not that the Fourier transform of the function $\Delta_n h$ for a given, summable function $h: \sfrac{1}{n} \bb Z^2 \to \bb R$ is given by:
\begin{equation}
\begin{split}\label{ft_of_delta_n}
\widehat{(\Delta_n h)}_n (k,\ell) = - n^2\Lambda \big( \tfrac{k}{n}, \tfrac{\ell}{n}\big) \widehat{h}_n(k,\ell),
\end{split}
\end{equation}
where
\begin{equation}
\begin{split}
 \Lambda \big( \tfrac{k}{n}, \tfrac{\ell}{n}\big)&= -\big( e^{\frac{2\pi i k}{n}}\!\! + e^{-\frac{2\pi i k}{n}}\!\! + e^{\frac{2 \pi i \ell}{n}} \!\!+ e^{-\frac{2\pi i \ell}{n}} \!\!-4\big)\\
		&=4 \left[ \sin^2\big(\tfrac{\pi k}{n}\big) + \sin^2\big(\tfrac{\pi \ell}{n}\big)\right].
\end{split}
\end{equation}
Similarly, the Fourier transform of ${\mc A}_n h$ is given by
\begin{equation}\label{ft_of_a_n}
\begin{split}
\widehat{({\mc A}_n h)}_n (k,\ell)=i\,n \,\Omega \big( \tfrac{k}{n}, \tfrac{\ell}{n}\big) \widehat{h}_n(k,\ell),
\end{split}
\end{equation}
 where
\begin{equation}
\begin{split}
i \,\Omega \big( \tfrac{k}{n}, \tfrac{\ell}{n}\big)&= e^{\tfrac{2\pi i k}{n}}\!\! + e^{\tfrac{2\pi i \ell}{n}}\!\! -e^{-\tfrac{2\pi i k}{n}}\!\! - e^{-\tfrac{2 \pi i \ell}{n}}\\
		&= 2\,i\,\big( \sin\big(\tfrac{2\pi k}{n}\big) + \sin \big(\tfrac{2 \pi \ell}{n} \big)\big).
\end{split}
\end{equation}
Note in particular that $\Omega(\frac{k}{n}, \frac{\ell}{n})$ is a real number.
Let us now compute the Fourier transform of the function $g_n= \nabla_n f \otimes \delta$ defined in (\ref{eq:3.9}):
\begin{equation}
\begin{split}
\widehat{g}_n(k,\ell)
		& = \cfrac{1}{n^2} \sum_{x,y\in \mathbb{Z}} \big[ \nabla_n f  \otimes \delta \big] \big( \tfrac{\vphantom{y}x}{n}, \tfrac{y}{n}\big) e^{ \tfrac{2i \pi(kx + \ell y)}{n}}\\
		&= - \cfrac{i n}{2} \Omega \big( \tfrac{k}{n}, \tfrac{\ell}{n}\big) {\widehat{f_n}} (k+ \ell).
\end{split}
\end{equation}
From the previous computations, we have that
\begin{equation}
\label{ecB.4}
\widehat{h}_n(k,\ell) =  \frac{1}{\sqrt n} \frac{i\,\Omega \big( \tfrac{k}{n}, \tfrac{\ell}{n}\big)}{ \gamma^{-1} n^b \Lambda\big( \tfrac{k}{n}, \tfrac{\ell}{n}\big) - i\, \Omega \big( \tfrac{k}{n}, \tfrac{\ell}{n}\big)} \; \widehat{f_n}(k+\ell).
\end{equation}
Our aim will be to study the behavior of $h_n$, as $n \to \infty$.}

\subsection{Proof of \eqref{eq:l31}}
\label{sec:proof_l3_1}

We want to show that
\begin{equation}
\|h_n\|^2_{2,n}:= \frac{1}{n^2} \sum_{x,y \in \bb Z} h_n\big(\tfrac{\vphantom{y}x}{n}, \tfrac{y}{n}\big)^2,
\end{equation}
vanishes, as $n\to\infty$. By Plancherel-Parseval's relation, Lemma \ref{lem:cov} \textcolor{black}{ and \eqref{ecB.4}},  we have that
\begin{equation*}
\begin{split}
\| h_n \|^2_{2,n}
		&=\iint_{[-\tfrac{n}{2}, \tfrac{n}{2}]^2} | {\widehat h}_n (k, \ell) |^2 dk d\ell \\
		&= \frac{1}{n} \iint_{[-\tfrac{n}{2}, \tfrac{n}{2}]^2}\frac{\Omega\big( \tfrac{k}{n}, \tfrac{\ell}{n}\big)^2 \, |{\widehat f_n} (k+\ell)|^2}{\gamma^{-2} n^{2b}\Lambda \big( \tfrac{k}{n}, \tfrac{\ell}{n}\big)^2 + \Omega \big( \tfrac{k}{n}, \tfrac{\ell}{n}\big)^2} \; dk d\ell.
\end{split}
\end{equation*}
Since 
\begin{equation}
\label{eq:est_omega}
\Omega \big( \tfrac{\xi -\ell}{n}, \tfrac{\ell}{n}\big)^2 \le 4 \Big| 1- e^{\tfrac{2 i \pi \xi}{n}} \Big|^2=16 \sin^2 \big(\tfrac{\pi\xi}{n}\big),
\end{equation}
last expression can be bounded from above by 
\begin{equation*}
\begin{split}
& \frac{16}{n} \int_{-n/2}^{n/2} \sin^2 ({\tfrac{\pi \xi}{n}})\, \big|{\widehat f}_n (\xi)\big|^2 \; \left[ \int_{-n/2}^{n/2} \frac{d\ell}{\gamma^{-2} n^{2b} \Lambda \big( \tfrac{\xi - \ell}{n}, \tfrac{\vphantom{\xi}\ell}{n}\big )^2 + \Omega \big( \tfrac{\xi - \ell}{n}, \tfrac{\vphantom{\xi}\ell}{n}\big)^2} \right] \; d\xi \\
		=& 16 n \int_{-1/2}^{1/2} \sin^2 (\pi y) |{\widehat f}_n (ny) |^2 {\widetilde W_n} (y) dy,
\end{split}
\end{equation*}
where
\begin{equation}
\label{eq:function_W}
\begin{split}
{\widetilde W_n} (y) =&  \int_{-1/2}^{1/2}  \frac{dx}{\gamma^{-2} n^{2b}\Lambda (y-x,x)^2 + \Omega (y-x, x)^2 } \\\le&  \int_{-1/2}^{1/2}  \frac{dx}{\gamma^{-2} \Lambda (y-x,x)^2 + \Omega (y-x, x)^2 }.
\end{split}
\end{equation}
Since by Lemma F.5 in \cite{BGJ}, the right hand side of (\ref{eq:function_W}) is of order $|y|^{-3/2}$ for $y\in[-\tfrac 12, \tfrac 12]$, then, from Lemma \ref{lem:sfp} we conclude that $\|h_n\|^2_{2,n}=O(\tfrac {1}{\sqrt n})$, which ends the proof of \eqref{eq:l31}.
\subsection{Proof of Lemma \ref{l3}}
\label{sec:proof_l3_2}
Let $q_n : {\sfrac{1}{n}} \ZZ \to \RR$ be the function defined by
\begin{equation}
q_n \big(\tfrac{x}{n}) = n^{a-b -3/2} \, {{\mc D}}_n h_n \,  \big(\tfrac{x}{n})
\end{equation}
and let $q:\RR \to \RR$ be defined by its Fourier transform ${\mc F} q$ given by
\begin{equation}
({\mc F} q)(\xi) = \int_{\RR} e^{2i \pi x \xi} q(x) dx= \cfrac{|\pi \xi|^{3/2}}{\sqrt {2\gamma}} e^{i {\rm{sgn}}(\xi) \tfrac{\pi}{4}}.
\end{equation}

\begin{lemma}
We have that
\begin{enumerate}[1.]
\item For $b \le 1/3$ and $a=3/2 +3b/2$,
\begin{equation}
\lim_{n \to + \infty} \frac{1}{n} \sum_{x \in \ZZ} \left[ q \big(\tfrac{x}{n} \big) -q_n \big( \tfrac{x}{n}\big)\right]^2 =0.
\end{equation}
\item For $b>1/3$ and $a=2$,
\begin{equation}
\lim_{n \to + \infty} \frac{1}{n} \sum_{x \in \ZZ} q^2_n \big(\tfrac{x}{n}\big) =0.
\end{equation}
\end{enumerate}
\end{lemma}

\begin{proof}
By following the proof of Lemma D.1 in \cite{BGJ} we have that 
\begin{equation*}
\begin{split}
{\widehat{q}_n} (\xi) =-  n^{a-b -3/2} \, \cfrac{i n}{2} \int_{-n/2}^{n/2} \Omega  \big( \tfrac{\xi -\ell}{n}, \tfrac{\ell}{n}\big)  {\widehat h_n} (\xi- \ell,\ell) \, d\ell.
\end{split}
\end{equation*}
By the explicit expression (\ref{ecB.4}) for ${\widehat h_n}$ we obtain that
\begin{equation*}
{\widehat{q_n}} (\xi) =  n^{a-b -3/2} \, \frac{\sqrt{n}}{2} \, \left[\int_{-n/2}^{n/2} \frac{\Omega  \big( \tfrac{\xi -\ell}{n}, \tfrac{\ell}{n}\big)^2}{ \gamma^{-1} n^b \Lambda \big( \tfrac{\xi - \ell}{n}, \tfrac{\ell}{n}\big )-i\Omega \big( \tfrac{\xi - \ell}{n}, \tfrac{\ell}{n}\big)} \, d\ell \right]\; {\widehat f_n} (\xi).
\end{equation*}
By the inverse Fourier transform we get that
\begin{equation*}
q_n\big(\tfrac{x}{n}\big)=2  n^{a-b}  \int_{-n /2}^{n/2}\; e^{- \tfrac{2i\pi \xi x}{n}} G_n \big( \tfrac{\xi}{n}\big)  {\widehat f}_n (\xi)   \; d\xi,
\end{equation*}
where 
\begin{equation}
\label{eq:fF}
G_n (y)= \frac{1}{4} \int_{-1/2}^{1/2} \cfrac{\Omega(y-z,z)^2 }{\gamma^{-1} n^b \Lambda (y-z,z) -i\, \Omega (y-z,z)}\, dz.
\end{equation}
By Lemma \ref{lem:G0345} we have that
\begin{equation}
\begin{split}
&G_n (y) = \cfrac{1}{n^{b/2} \sqrt{2 \gamma}} |\sin(\pi y)|^{3/2} \; e^{i {\rm{sgn}} (y) \tfrac{\pi}{4}} + {\mc O} ( \sin^2 (\pi y) ) \quad \text{if} \quad b<1, \\
& G_n (y) = {\mc O} \Big( n^{-b} |\sin (\pi y)|\Big)\quad \text{if} \quad b\ge1.
\end{split}
\end{equation}
By using Lemma \ref{lem:sfp}, we obtain easily the result. 
\end{proof}

\section{Proof of Lemma \ref{lem:5,7.}}

We start by computing the Fourier transform of $v_n$. Recall that $v_n $ is solution of \eqref{eq:poisson-v}. Applying the Fourier transform we get that
\begin{equation*}
\widehat{\Delta_n v_n} (k,\ell) + \gamma n^{1-b}\widehat{ \mc A_n v_n}(k,\ell)  =\widehat{ {\widetilde {\mc D}}_n h_n}.
\end{equation*}
Note that since $h_n$ does not depend on $a$ it is the same for $v_n$.
From \eqref{ft_of_delta_n} and \eqref{ft_of_a_n} the left hand side of the previous display is given by
\begin{equation*}
-n^{2}\Lambda\big(\tfrac{k}{n}, \tfrac{\ell}{n}\big) \widehat{v_n} (k,\ell) + i\gamma n^{2-b}\Omega \big(\tfrac{k}{n}, \tfrac{\ell}{n}\big)\widehat{ v_n}\big(k,\ell) .
\end{equation*}
Now we need to compute the Fourier transform of 
${\widetilde{\mc D}}_n h_n$.
By a simple computation we get that
\begin{equation}\label{eq:D_tilde}
\begin{split}
\widehat{\widetilde{\mc D}_n{h}_n}(k,l) =& \sum_{x\in\mathbb{Z}}\Big\{ {\widetilde{\mc E}}_n h  \big(\tfrac{x}{n}\big)-\Big(\tfrac{1-k}{2}\Big) {\widetilde{\mc F}}_n h\big(\tfrac{x}{n}\big)\Big\}\Big\{e^{\tfrac{2\pi i (kx + \ell (x+1))}{n}}+e^{\tfrac{2 i \pi (k(x+1)+\ell x)}{n}}\Big\}\\
=&n \Big\{e^{\tfrac{2i\pi\ell}{n}}+e^{\tfrac{2i\pi k}{n}}\Big\}\Big\{\widehat{{\widetilde{\mc E}}_n h }\,  (k+\ell)-\Big(\tfrac{1-k}{2}\Big)\widehat{{\widetilde{\mc F}}_n h}\, (k+\ell)\Big\},
\end{split}
\end{equation}
\textcolor{black}{where ${\widetilde{\mc E}}_n h$ and ${\widetilde{\mc F}}_n h$ were given in \eqref{eq:En} and \eqref{eq:Fn}, respectively.} 
From the previous computations we conclude  that the Fourier transform ${\widehat v}_n$ is given by
\begin{equation}
\label{FT of vn}
\begin{split}
{\widehat v}_n (k, \ell) & = -\cfrac{1}{n} \, \cfrac{e^{ \tfrac{2i \pi k}{n}} +e^{\tfrac{2i \pi \ell}{n}} }{  \Lambda\big( \tfrac{k}{n}, \tfrac{\ell}{n}\big) -i\gamma n^{-b}\, \Omega\big( \tfrac{k}{n}, \tfrac{\ell}{n}\big) } \, \Big\{\widehat{{\widetilde{\mc E}}_n h }\, (k+\ell)-\Big(\tfrac{1-k}{2}\Big)\widehat{{\widetilde{\mc F}}_n h }\, (k+\ell)\}\\
&=- \gamma^{-1} n^{b-1}  \, \cfrac{e^{ \tfrac{2i \pi k}{n}} +e^{\tfrac{2i \pi \ell}{n}} }{  \gamma^{-1} n^b \Lambda\big( \tfrac{k}{n}, \tfrac{\ell}{n}\big) -i \, \Omega\big( \tfrac{k}{n}, \tfrac{\ell}{n}\big) } \, \Big\{\widehat{{\widetilde{\mc E}}_n h }\, (k+\ell)-\Big(\tfrac{1-k}{2}\Big)\widehat{{\widetilde{\mc F}}_n h }\, (k+\ell)\}.
\end{split}
\end{equation}
By Lemma \ref{lem:cov} \textcolor{black}{and \eqref{eq:En}} we have that the Fourier transform of ${{\widetilde{\mc E}}_n h }\, $ is given by
\begin{equation}\label{exp_imp_1}
\begin{split}
\widehat{{\widetilde{\mc E}}_n h }\,  (\xi) &= \frac{1}{n} \sum_{x\in\mathbb{Z}} e^{ \tfrac{2i\pi \xi x}{n}} \Big(h_n \big( \tfrac{x}{n} , \tfrac{x+1}{n} \big) -h_n \big( \tfrac{x}{n} , \tfrac{x}{n} \big)\Big)\\
&= \frac{1}{n} \sum_{x\in\mathbb{Z}} e^{ \tfrac{2i\pi \xi x}{n}} \iint_{[-\tfrac{n}{2}, \tfrac{n}{2}]^2} {\widehat h}_n (k,\ell) e^{- \tfrac{2 i \pi (k+\ell) x}{n}} \big\{ e^{- \tfrac{2i\pi\ell}{n}} -1 \big\} \; dk d\ell \\
&= \frac{1}{n} \sum_{x\in\mathbb{Z}} e^{ \tfrac{2i\pi \xi x}{n}} \int_{-n/2}^{n/2} \; e^{- \tfrac{2 i \pi u x}{n}} \left\{ \int_{-n/2}^{n/2}  {\widehat h}_n (u-\ell,\ell)\big\{ e^{- \tfrac{2i\pi \ell}{n}} -1 \big\} d\ell \right\} du\\
&=  \int_{-n/2}^{n/2}  {\widehat h}_n (\xi-\ell,\ell)\big\{ e^{- \tfrac{2i\pi \ell}{n}} -1 \big\} \, d\ell.
\end{split}
\end{equation}
In the last line we used the inverse Fourier transform.
By (\ref{ecB.4}) we get that
\begin{equation}\label{ft_of_E_n_h}
\begin{split}
\widehat{{\widetilde{\mc E}}_n h }\,  (\xi) &= -\frac{1}{\sqrt{n}} \, {\widehat f}_n (\xi) \,  \int_{-n/2}^{n/2}  \frac{\big(1- e^{- \tfrac{2i\pi \ell}{n}} \big)\,  i\, \Omega \big( \tfrac{\xi - \ell}{n}, \tfrac{\vphantom{\xi}\ell}{n}\big)} { \gamma^{-1} n^b\Lambda\big( \tfrac{\xi - \ell}{n}, \tfrac{\vphantom{\xi}\ell}{n}\big) - i \, \Omega\big( \tfrac{\xi - \ell}{n}, \tfrac{\vphantom{\xi}\ell}{n}\big)}\, d\ell\\
&= -\sqrt{n} I_n \big( \tfrac{\xi}{n} \big) {\widehat f}_n (\xi),
\end{split}
\end{equation}
where the function $I_n $ is defined by
\begin{equation}
\label{eq:int_I}
I_n (y)= \int_{-1/2}^{1/2} \cfrac{(1-e^{-2i \pi x}) \, i \, \Omega (y-x,x)}{\gamma^{-1} n^b\Lambda(y-x,x) -i \Omega(y-x,x)} dx.
\end{equation}
Again, by Lemma \ref{lem:cov} \textcolor{black}{and \eqref{eq:Fn}}  the Fourier transform of $\widetilde{\mc F}_n h$ is given by
\begin{equation}\label{exp_imp_2}
\begin{split}
\widehat{{\widetilde{\mc F}}_n h }\,  (\xi) &= \frac{1}{n} \sum_{x\in\mathbb{Z}} e^{ \tfrac{2i\pi \xi x}{n}} \Big(h_n \big( \tfrac{x+1}{n} , \tfrac{x+1}{n} \big) -h_n \big( \tfrac{x}{n} , \tfrac{x}{n} \big)\Big)\\
&=\Big( e^{ -\tfrac{2i\pi \xi }{n}}-1\Big)\frac{1}{n} \sum_{x\in\mathbb{Z}} e^{ \tfrac{2i\pi \xi x}{n}} h_n \big( \tfrac{x}{n} , \tfrac{x}{n} \big)\\
&= \Big( e^{ -\tfrac{2i\pi \xi }{n}}-1\Big)\frac{1}{n} \sum_{x\in\mathbb{Z}} e^{ \tfrac{2i\pi \xi x}{n}} \iint_{[-\tfrac{n}{2}, \tfrac{n}{2}]^2} {\widehat h}_n (k,\ell) e^{- \tfrac{2 i \pi (k+\ell) x}{n}} \; dk d\ell \\
&=\Big( e^{ -\tfrac{2i\pi \xi }{n}}-1\Big) \frac{1}{n} \sum_{x\in\mathbb{Z}} e^{ \tfrac{2i\pi \xi x}{n}} \int_{-n/2}^{n/2} \; e^{- \tfrac{2 i \pi u x}{n}} \left\{ \int_{-n/2}^{n/2}  {\widehat h}_n (u-\ell,\ell) d\ell \right\} du\\
&= \Big( e^{ -\tfrac{2i\pi \xi }{n}}-1\Big)
 \int_{-n/2}^{n/2}  {\widehat h}_n (\xi-\ell,\ell)\, d\ell.
\end{split}
\end{equation}
By (\ref{ecB.4}) we get
\begin{equation}\label{ft_of_F_n_h}
\begin{split}
\widehat{{\widetilde{\mc F}}_n h }\,  (\xi) &= -\frac{1}{\sqrt{n}}\big(1- e^{- \tfrac{2i\pi \xi}{n}} \big) \, {\widehat f}_n (\xi) \,  \int_{-n/2}^{n/2}  \frac{ i\, \Omega \big( \tfrac{\xi - \ell}{n}, \tfrac{\vphantom{\xi}\ell}{n}\big)} { \gamma^{-1} n^b\Lambda\big( \tfrac{\xi - \ell}{n}, \tfrac{\vphantom{\xi}\ell}{n}\big) - i \, \Omega\big( \tfrac{\xi - \ell}{n}, \tfrac{\vphantom{\xi}\ell}{n}\big)}\, d\ell\\
&= -\sqrt{n}\big(1- e^{- \tfrac{2i\pi \xi}{n}} \big) \widetilde{I_n}\big( \tfrac{\xi}{n} \big) {\widehat f}_n (\xi),
\end{split}
\end{equation}
where the function $\widetilde{I_n}$ is defined by
\begin{equation}
\label{eq:int_tilde_I}
{\widetilde I_n} (y)= \int_{-1/2}^{1/2}\frac{i \, \Omega (y-x,x)}{\gamma^{-1} n^b\Lambda(y-x,x) -i \Omega(y-x,x)} dx.
\end{equation}

\subsection{Proof of \eqref{est vn}} \label{subsec:secondterm}

In order to compute the discrete $\bb L^2$ norm of $v_n$ we use the Plancherel-Parseval's relation, \textcolor{black}{\eqref{FT of vn}}, Lemma \ref{lem:cov}
and we have that 
\begin{equation*}
\begin{split}
\|v_n\|_{2,n}^2
		&= \iint_{[-\tfrac{n}{2},\tfrac{n}{2}]^2} |{\widehat v_n} (k, \ell)|^2 dk d\ell\\
		&= \frac{1}{n^2} \int_{-n/2}^{n/2} |\widehat{{\widetilde{\mc E}}_n h }\,(\xi)-\Big(\tfrac{1-k}{2}\Big)\widehat{{\widetilde{\mc F}}_n h }\,(\xi)|^2 \,  \int_{-n/2}^{n/2} \left|\cfrac{e^{ \tfrac{2i \pi(\xi - \ell)}{n} } +e^{ \tfrac{2 i \pi \ell}{n} }} {\Lambda\big(\tfrac{\xi -\ell}{n}, \tfrac{\vphantom{\xi}\ell}{n} \big) -i\gamma n^{-b} \Omega\big(\tfrac{\xi -\ell}{n}, \tfrac{\vphantom{\xi}\ell}{n} \big) }\,  \right|^2 d\ell d\xi\\
		&\le \frac{C}{n} \int_{-n/2}^{n/2} \Big| \widehat{{\widetilde{\mc E}}_n h }\, (\xi)-\Big(\tfrac{1-k}{2}\Big)\widehat{{\widetilde{\mc F}}_n h }\,(\xi)\Big|^2 W_n \big( \tfrac{\xi}{n}\big) d\xi,
\end{split}
\end{equation*}
where 
\begin{equation}
\label{eq:int_omega}
W_n(y)=\int_{-1/2}^{1/2}\frac{dx}{\Lambda(y-x,x)^2+\gamma^2 n^{-2b}\Omega(y-x,x)^2}.
\end{equation}
For $0<y<1/2$ we observe that since $a \le 2$, we can bound from above $W_n (y)$ by
\begin{equation}
W_n (y) \le n^{2b} \int_{-1/2}^{1/2}\frac{dx}{\Lambda(y-x,x)^2+\gamma^2 \Omega(y-x,x)^2}.
\end{equation}
This integral has been estimated in \textcolor{black} {Lemma F.5 of \cite{BGJ}} and is of order $|y|^{-3/2}$ for $y \to 0$. Therefore we have that
\begin{equation}
\label{eq:E10}
|W_n (y)| \le C \; n^{2b} \; |y|^{-3/2}.
\end{equation}


By the triangular inequality, in order to finish the proof,  we are reduced to show that
\begin{equation}\label{last_terms}
\begin{split}
 \frac{1}{n} \int_{-n/2}^{n/2}| \widehat{{\widetilde{\mc E}}_n h }\,(\xi)|^2 W_n \big( \tfrac{\xi}{n}\big) d\xi \quad \textrm{and} \quad
  \frac{1}{n} \int_{-n/2}^{n/2}| \widehat{{\widetilde{\mc F}}_n h }\,(\xi)|^2 W_n \big( \tfrac{\xi}{n}\big) d\xi
\end{split}
\end{equation} 
vanish as $n\to\infty$.
By \eqref{ft_of_E_n_h} the term at the left hand side of the previous display is equal to 
\begin{equation*}
\begin{split}
\int_{-n/2}^{n/2}|\widehat{f}_n(\xi)|^2 | I_n (\tfrac{\xi}{n})|^2 W_n \big( \tfrac{\xi}{n}\big) d\xi={n} \int_{-1/2}^{1/2}|\widehat{f}_n(ny)|^2 | I_n (y)|^2 W_n (y) dy.
\end{split}
\end{equation*}
By Lemma \ref{lem:in18} and Lemma \ref{lem:sfp} we have
\begin{equation*}
\begin{split}
&{n} \int_{-1/2}^{1/2}|\widehat{f}_n(ny)|^2 | I_n (y)|^2 W_n (y) dy \\
&\le C {n} \int_{-1/2}^{1/2} \tfrac{1}{1+|ny|^p} \tfrac{|y|^{3/2}}{|y| + n^{-b}} \, dy= 2C\, \tfrac{1}{\sqrt n} \,  \int_0^{n/2} \tfrac{1}{1+|z|^p} \cfrac{|z|^{3/2}}{|z| + n^{1-b}} \, dz\\
& \le 2C\, \tfrac{1}{\sqrt n} \,  \int_0^{n/2} \tfrac{|z|^{1/2}}{1+|z|^p} \, dz
\end{split}
\end{equation*}
which goes to $0$ as $n \to \infty$ by choosing $p  \ge 3/2$.
Finally, by \eqref{ft_of_F_n_h}, Lemma \ref{lem:in18} and Lemma \ref{lem:sfp}, the term at the right hand side of \eqref{last_terms} is bounded from above as
\begin{equation*}
\begin{split}
 & \int_{-n/2}^{n/2}\Big| 1-e^{\tfrac{2\pi i\xi}{n}}\Big|^2 |{\widetilde I_n} (\tfrac{\xi}{n})|^2|\widehat{f}_n(\tfrac{\xi}{n})|^2 W_n \big( \tfrac{\xi}{n}\big) d\xi\\
&= n \int_{-1/2}^{1/2}|\sin(\pi y)|^2|\widehat{f}_n(ny)|^2 | \widetilde {I_n}(y)|^2 W_n (y) dy\\
&\le C\, \tfrac{1}{\sqrt n} \,  \int_0^{n/2} \tfrac{|z|^{1/2}}{1+|z|^p} \, dz
\end{split}
\end{equation*}
which goes to $0$ as $n \to \infty$ by choosing $p  \ge 3/2$.

\subsection{Proof of \eqref{est der vn}}
\textcolor{black}{In order to compute the discrete $\bb L^2$ norm of  $ \mc D_n v_n$ we first  note that by simple computations together with Lemma \ref{lem:cov} we have that (see equation (E.7) in \cite{BGJ})}
\begin{equation}\label{eq:d_n}
\begin{split}
{\widehat{{\mc D}_n v_n}}(\xi) =n\Big(1-e^{\tfrac{2i\pi \xi }{n}}\Big)\; \int_{-\tfrac{n}{2}}^{\tfrac{n}{2}} {\widehat v_n} (\xi-\ell, \ell)e^{-\tfrac{2i\pi\ell}{n}} d\ell.
	\end{split}
\end{equation}
By \eqref{FT of vn} last expression equals to 
		\begin{equation}
\begin{split}
&{\widehat{{\mc D}_n v_n}}(\xi)\\
		&=-\big(1-e^{\tfrac{2i\pi \xi }{n}}\big)\Big\{\widehat{{\widetilde{\mc E}}_n h }\,(\xi)-\Big(\tfrac{1-k}{2}\Big)\widehat{{\widetilde{\mc F}}_n h }\,(\xi)\Big\}\int_{-\tfrac{n}{2}}^{\tfrac{n}{2}}\frac{1+e^{ \tfrac{2i \pi(\xi-2\ell)}{n}}}{ \Lambda \big( \tfrac{\xi-\ell}{n}, \tfrac{\vphantom{\xi}\ell}{n}\big) -i\gamma n^{-b} \Omega \big( \tfrac{\xi-\ell}{n}, \tfrac{\vphantom{\xi}\ell}{n}\big)}d\ell\\
		&=-{n}\big(1-e^{\tfrac{2i\pi \xi }{n}}\big)\Big\{\widehat{{\widetilde{\mc E}}_n h }\,(\xi)-\Big(\tfrac{1-k}{2}\Big)\widehat{{\widetilde{\mc F}}_n h }\,(\xi)\Big\} J_n \big (\tfrac{\xi}{n}\big),
\end{split}
\end{equation}
where  $J_n$ is given by
\begin{equation}
\label{eq:int_J}
J_n (y) = \int_{-1/2}^{1/2} \cfrac{1+e^{2i \pi (y-2x)}}{\Lambda(y-x,x)-i\gamma n^{-b}\Omega(y-x,x) } dx.
\end{equation}
Now, by using \eqref{ft_of_E_n_h} and \eqref{ft_of_F_n_h} we get
\begin{equation}
{\widehat{{\mc D}_n v_n}}(\xi)=\frac{n^{3/2}}{2}\big( 1- e^{ \tfrac{2i \pi \xi}{n}} \big)  {\widehat f}_n (\xi) I_n  \big (\tfrac{\xi}{n}\big) J_n \big (\tfrac{\xi}{n}\big)-4n^{3/2}\Big(\tfrac{1-k}{2}\Big) \sin^2(\tfrac{\pi \xi}{n}) \widetilde {I_n}(\tfrac{\xi}{n})\widehat{f}_n(\tfrac{\xi}{n})J_n(\tfrac{\xi}{n})
\end{equation}
where $I_n$ is defined by (\ref{eq:int_I}). Finally, by 
the Plancherel-Parseval's relation we have that
\begin{equation}
\label{eq:tronconeuse_0}
\begin{split}
\|n^{a-b -3/2} \,  \mc D_nv_n\|^2_{2,n}&\leq C n^{2(a-b)} \int_{-n/2}^{n/2} \sin^2 \big (\pi \tfrac{\xi}{n}\big )\big| {\widehat f}_n (\xi) \big|^2 \big| I_n \big (\tfrac{\xi}{n}\big) \big|^2 \big| J_n \big (\tfrac{\xi}{n}\big) \big|^2    d\xi
\\
&+C n^{2(a-b)} \int_{-n/2}^{n/2} |\sin \big (\pi \tfrac{\xi}{n}\big )|^4\big| {\widehat f}_n (\xi) \big|^2 \big|\widetilde{I_n} \big (\tfrac{\xi}{n}\big) \big|^2 \big| J_n \big (\tfrac{\xi}{n}\big) \big|^2    d\xi\\
&=C  n^{2(a-b)+1} \int_{-1/2}^{1/2} |\sin (\pi y)|^2 | I_n (y)|^2 | J_n (y)|^2 |\widehat f_n (ny)|^2 dy\\
&+C n^{2(a-b)+1} \int_{-1/2}^{1/2} |\sin \big (\pi y\big )|^4 \big|\widetilde{I_n} (y) \big|^2 \big| J_n (y) \big|^2 \big| {\widehat f}_n (ny) \big|^2   dy.
\end{split}
\end{equation}
By using Lemma \ref{lem:j_nint}, Lemma \ref{lem:in18} and Lemma \ref{lem:sfp}, choosing a $p$ sufficiently large, we can bound the first term at the right hand side of (\ref{eq:tronconeuse_0}) by a constant times
\begin{equation}
\label{eq:tournedos}
\begin{split}
 n^{2a -4b +1} \int_0^{1/2} \tfrac{y^4}{1+ (ny)^p} \tfrac{1}{(|y| +n^{-b})^2} \, dy = n^{2a -2 -4b} \int_{0}^{\infty} \tfrac{z^4}{1+z^p} \tfrac{1}{(z + n^{1-b})^2}\, dz. 
\end{split}
\end{equation}
If $b \le 1$, the previous integral  is bounded from above by 
$$n^{2a -4 -2b} \int_{0}^{\infty} \tfrac{z^4}{1+z^p}\, dz$$
which goes to $0$ since $a=\inf (3/2 -3/2b, 2)$.  If $b > 1$, the integral (\ref{eq:tournedos})  is bounded from above by 
$$n^{2a -2 -4b} \int_{0}^{\infty} \tfrac{z^2}{1+z^p}\, dz$$
which goes to $0$ since $a=2$ in this case. 
 
The second term on the right hand side  of (\ref{eq:tronconeuse_0}) is proved to go to $0$ similarly by using Lemma \ref{lem:j_nint}, Lemma \ref{lem:in18} and Lemma \ref{lem:sfp}. This completes the proof of \eqref{est der vn}.

\section{Proof of \eqref{new_da}}

\begin{lemma}\label{lem:new_da}
Let $v_n:\frac{1}{n}\bb Z^2\to\bb R$ be the solution of  
\eqref{eq:poisson-v}
and let $a=\inf(3/2+3b/2, 2)$ and $b\in(0,1)$. For any $T>0$ we have that
 \begin{equation*}
\lim_{n\to\infty} \bb E\Big[\Big(\int_0^T Q_t^n (n^{a-2}\widetilde{\mc D}_n v_n)\,dt\Big)^2 \Big]=0.
 \end{equation*}
\end{lemma}

\begin{proof}
In order to prove the result we do the following. We use again (\ref{ec3.15}) for the solution $w_n$ of the Poisson equation
\begin{equation}
\label{eq:poisson-w}
n^{a-2} \Delta_n w_n \big(\tfrac{\vphantom{y}x}{n},\tfrac{y}{n}\big) + \gamma n^{a-1-b} \mc A_n w_n\big(\tfrac{\vphantom{y}x}{n},\tfrac{y}{n}\big)  =n^{a-2} {\widetilde {\mc D}}_n v_n.
\end{equation}

Then, by integrating in time \eqref{ec3.15} we have
\begin{equation*}
\begin{split}
\int_0^T Q_t^n \big( n^{a-2} {\widetilde{\mc D}}_n v_n \big) dt &= 2\gamma n^{a-b-3/2}  \int_0^T {\mc S}_t^n ({\mc D}_n w_n) dt -2 \int_{0}^T {Q}_t^n  \big( n^{a-2} {\widetilde{\mc D}}_n w_n \big) dt \\
&+ Q_T^n (w_n) -Q_0^n (w_n).
\end{split}
\end{equation*}
Now, by using repeatedly the inequality $(x+y)^2\leq 2x^2+2y^2$ in order to conclude we have to show that 
 \begin{equation*}
\bb E\Big[\Big(n^{a-b-3/2}  \int_0^T {\mc S}_t^n ({\mc D}_n w_n)\,dt\Big)^2 \Big],
 \end{equation*} 
and 
\begin{equation*}
\bb E\Big[\Big( Q_T^n (w_n) -Q_0^n (w_n)\Big)^2 \Big]
 \end{equation*}
 and
 \begin{equation*}
\bb E\Big[\Big(\int_{0}^T {Q}_t^n  \big( n^{a-2} {\widetilde{\mc D}}_n w_n \big)\, dt \Big)^2 \Big]
 \end{equation*}
 vanish as $n$ goes to infinity. The first display above, by \eqref{apriori_bound_S} and \eqref{est der wn}, vanishes, as $n\to\infty$. Similarly the second (resp. third display), by \eqref{apriori_bound_Q} and \eqref{est wn} (resp. \eqref{est tilde d wn}), vanishes as $n\to\infty$.
\end{proof}

Therefore the previous lemma depends on the following estimates on $w_n$.
\begin{lemma}
\label{lem:est wn}
Let $a=\inf(3/2+3b/2, 2)$ and $b\in (0,1)$. The solution $w_n$ of \eqref{eq:poisson-w} satisfies
\begin{equation}
\label{est wn}
\lim_{n \to \infty}\|w_n\|_{2,n}^2=0,
\end{equation}
\begin{equation}\label{est der wn}
\lim_{n \to \infty}\|n^{a-b -3/2} \,  \mc D_nw_n\|^2_{2,n}=0,
\end{equation}
\begin{equation}\label{est tilde d wn}
\lim_{n \to \infty} \|n^{a-2} \,  \tilde{\mc D_n}w_n\|^2_{2,n}=0.
\end{equation}
\end{lemma}
\begin{proof}
We start by computing the Fourier transform of $w_n$. Repeating the computations done for  \eqref{FT of vn} and recalling that  $w_n $ is solution of \eqref{eq:poisson-w}, we obtain 
\begin{equation}
\label{FT of wn}
\begin{split}
{\widehat {w}}_n (k, \ell) & = -\cfrac{1}{n} \, \cfrac{e^{ \tfrac{2i \pi k}{n}} +e^{\tfrac{2i \pi \ell}{n}} }{  \Lambda\big( \tfrac{k}{n}, \tfrac{\ell}{n}\big) -i\gamma n^{-b}\, \Omega\big( \tfrac{k}{n}, \tfrac{\ell}{n}\big) } \, \Big\{\widehat{\widetilde{ \mc{E}_n}v_n}(k+\ell)-\Big(\tfrac{1-k}{2}\Big){\widehat{\widetilde{ \mc{F}_n}v_n}}(k+\ell)\Big\}
\end{split}
\end{equation}
\textcolor{black}{where ${\widetilde{\mc E}}_n v_n$ and ${\widetilde{\mc F}}_n v_n$
are defined as in \eqref{eq:En} and \eqref{eq:Fn}
 with $h_n$ replaced with $v_n$.}\\

$\bullet$ We start by proving \eqref{est wn}. As in Section \ref{subsec:secondterm} we have that 
\begin{equation*}
\begin{split}
&\|w_n\|_{2,n}^2
		\le \frac{C}{n} \int_{-n/2}^{n/2} \Big|\widehat{\widetilde{ \mc{E}_n}v_n}(\xi)-\Big(\tfrac{1-k}{2}\Big)\widehat{\widetilde{ \mc{F}_n}v_n}(\xi)\Big|^2 W_n \big( \tfrac{\xi}{n}\big) d\xi
\end{split}
\end{equation*}
where $W_n$ is given in \eqref{eq:int_omega}.
By the triangular inequality, in order to finish the proof,  we are reduced to show that
\begin{equation}\label{last_terms_new}
\begin{split}
 \frac{1}{n} \int_{-n/2}^{n/2}|\widehat{\widetilde{ \mc{E}_n}v_n}(\xi)|^2 \, W_n \big( \tfrac{\xi}{n}\big) d\xi \quad \textrm{and} \quad
  \frac{1}{n} \int_{-n/2}^{n/2}| \widehat{\widetilde{ \mc{F}_n}v_n}(\xi)|^2\,  W_n \big( \tfrac{\xi}{n}\big) d\xi
\end{split}
\end{equation} 
vanish as $n\to\infty$.
Now we compute  the  Fourier transform of the previous functions. As
in \eqref{exp_imp_1} and using \eqref{FT of vn} we have that
\begin{equation}\label{ft_of_E_n_v}
\begin{split}
\widehat{\widetilde{ \mc{E}_n}v_n} (\xi) &= - \cfrac{1}{n} \,\int_{-n/2}^{n/2}  \cfrac{e^{ \tfrac{2i \pi (\xi-\ell)}{n}} +e^{\tfrac{2i \pi \ell}{n}} }{  \Lambda\big( \tfrac{\xi-\ell}{n}, \tfrac{\ell}{n}\big) -i\gamma n^{-b}\, \Omega\big( \tfrac{\xi-\ell}{n}, \tfrac{\ell}{n}\big) } \; \widehat{\widetilde{ \mc{E}_n}h_n}(\xi)\big\{ e^{- \tfrac{2i\pi \ell}{n}} -1 \big\} \, d\ell\\
&+ \Big(\tfrac{1-\kappa}{2}\Big) \, \cfrac{1}{n} \,\int_{-n/2}^{n/2}  \cfrac{e^{ \tfrac{2i \pi (\xi-\ell)}{n}} +e^{\tfrac{2i \pi \ell}{n}} }{  \Lambda\big( \tfrac{\xi-\ell}{n}, \tfrac{\ell}{n}\big) -i\gamma n^{-b}\, \Omega\big( \tfrac{\xi-\ell}{n}, \tfrac{\ell}{n}\big) }\;  \widehat{\widetilde{ \mc{F}_n}h_n}(\xi)\big\{ e^{- \tfrac{2i\pi \ell}{n}} -1 \big\} \, d\ell\\
&=\sqrt n \,  K_n(\tfrac{\xi}{n})\; \Big\{ I_n(\tfrac{\xi}{n})-\Big(\tfrac{1-\kappa}{2}\Big)\big(1-e^{-\frac{2\pi i \xi}{n}}\big)\, {\widetilde I}(\tfrac{\xi}{n})\Big\} \; \widehat{f}_n(\xi),
\end{split}
\end{equation}
where above we used \eqref{ft_of_E_n_h}, \eqref{ft_of_F_n_h} and where $K_n$ is given by
\begin{equation}
\label{eq:int_K}
K_n(y) = \int_{-1/2}^{1/2} \cfrac{(e^{-2i \pi x} -1) (e^{2i \pi (y-x)}+ e^{2i \pi x}  )}{\Lambda(y-x,x) -i\gamma n^{-b} \Omega(y-x,x) }\, dx.
\end{equation}
Now, as in  \eqref{exp_imp_2}  and using \eqref{FT of vn} we have that
\begin{equation}\label{ft_of_F_n_v}
\begin{split}
&{\widehat {\widetilde{\mc F}_n v}}_n (\xi) =    -\cfrac{1}{n} \, \Big( e^{ -\tfrac{2i\pi \xi }{n}}-1\Big)\int_{-n/2}^{n/2} \, \cfrac{e^{ \tfrac{2i \pi (\xi-\ell)}{n}} +e^{\tfrac{2i \pi \ell}{n}} }{  \Lambda\big( \tfrac{\xi-\ell}{n}, \tfrac{\ell}{n}\big) -i\gamma n^{-b}\, \Omega\big( \tfrac{\xi-\ell}{n}, \tfrac{\ell}{n}\big) }\; {\widehat {\widetilde{\mc E}_n h_n}} (\xi) \, d\ell\\
&\quad \quad  + \cfrac{1}{n}\Big(\tfrac{1-\kappa}{2}\Big) \Big( e^{ -\tfrac{2i\pi \xi }{n}}-1\Big)\int_{-n/2}^{n/2}  \, \cfrac{e^{ \tfrac{2i \pi (\xi-\ell)}{n}} +e^{\tfrac{2i \pi \ell}{n}} }{  \Lambda\big( \tfrac{\xi-\ell}{n}, \tfrac{\ell}{n}\big) -i\gamma n^{-b}\, \Omega\big( \tfrac{\xi-\ell}{n}, \tfrac{\ell}{n}\big) } \; {\widehat {\widetilde{\mc F}_n h_n}} (\xi) \, d\ell\\
&= \sqrt n \; \big( e^{ -\tfrac{2i\pi \xi }{n}}-1\big)\; \widetilde{K_n} (\tfrac{\xi}{n})\; \Big\{I_n(\tfrac{\xi}{n})-\Big(\tfrac{1-k}{2}\Big)\Big(1-e^{-\frac{2\pi i \xi}{n}}\Big)\, {\widetilde I_n}(\tfrac{\xi}{n})\Big\} \, \widehat{f}_n(\xi),
\end{split}
\end{equation}
where above we used \eqref{ft_of_E_n_h}, \eqref{ft_of_F_n_h} and $\widetilde{K_n}$ is given by
\begin{equation}
\label{eq:int_K}
\widetilde{K_n} (y) = \int_{-1/2}^{1/2} \cfrac{ e^{2i \pi (y-x)}+ e^{2i \pi x}  }{\Lambda(y-x,x) -i\gamma n^{-b} \Omega(y-x,x) }\, dx.
\end{equation}
Now we estimate the term at the left hand side of \eqref{last_terms}, which,
by \eqref{ft_of_E_n_v} and the triangular inequality, can be bounded from above by the sum of the two terms below. The first one is equal to
\begin{equation*}
\begin{split}
&\int_{-n/2}^{n/2}|\widehat{f}_n(\xi)|^2 | I_n (\tfrac{\xi}{n})|^2  | K_n (\tfrac{\xi}{n})|^2W_n \big( \tfrac{\xi}{n}\big) d\xi\\
=&{n} \int_{-1/2}^{1/2}|\widehat{f}_n(ny)|^2 | I_n (y)|^2| K_n (y)|^2 W_n (y) dy,
\end{split}
\end{equation*}
which 
by Lemma \ref{lem:in18}, Lemma \ref{lem:K_nint}, Lemma \ref{lem:sfp} and (\ref{eq:E10}) is bounded from above by
\begin{equation*}
\begin{split}
& C {n} \int_{-1/2}^{1/2} \tfrac{1}{1+|ny|^p} \tfrac{|y|^{5/2}}{(|y| + n^{-b})^2} \, dy= 2C\, \tfrac{1}{\sqrt n} \,  \int_0^{n/2} \tfrac{1}{1+|z|^p} \cfrac{|z|^{5/2}}{(|z| + n^{1-b})^2} \, dz\\
& \le 2C\, \tfrac{1}{\sqrt n} \,  \int_0^{n/2} \tfrac{|z|^{1/2}}{1+|z|^p} \, dz
\end{split}
\end{equation*}
and goes to $0$ as $n \to \infty$ by choosing $p  \ge 3/2$. The second one is 
\begin{equation*}
\begin{split}
&\int_{-n/2}^{n/2}|\widehat{f}_n(\xi)|^2 | \widetilde{I_n} (\tfrac{\xi}{n})|^2 \Big|1-e^{-\tfrac{2\pi i\xi}{n}}\Big|^2 | K_n (\tfrac{\xi}{n})|^2W_n \big( \tfrac{\xi}{n}\big) d\xi\\=&{n} \int_{-1/2}^{1/2}|\widehat{f}_n(ny)|^2 | \widetilde{I_n} (y)|^2 \Big|1-e^{-{2\pi  iy}}\Big|^2| K_n (y)|^2 W_n (y) dy,
\end{split}
\end{equation*}
which Lemma \ref{lem:in18}, Lemma \ref{lem:K_nint}, Lemma \ref{lem:sfp} and (\ref{eq:E10}) is bounded from above by
\begin{equation*}
\begin{split}
& C {n} \int_{-1/2}^{1/2} \tfrac{1}{1+|ny|^p} \tfrac{|y|^{5/2}}{(|y| + n^{-b})^2} \, dy= 2C\, \tfrac{1}{\sqrt n} \,  \int_0^{n/2} \tfrac{1}{1+|z|^p} \cfrac{|z|^{5/2}}{(|z| + n^{1-b})^2} \, dz\\
& \le 2C\, \tfrac{1}{\sqrt n} \,  \int_0^{n/2} \tfrac{|z|^{1/2}}{1+|z|^p} \, dz
\end{split}
\end{equation*}
and goes to $0$ as $n \to \infty$ by choosing $p  \ge 3/2$. Therefore we have shown that the first term in \eqref{last_terms} goes to $0$ as $n$ goes to infinity. The estimate for the term at the right hand side of \eqref{last_terms} is similar and this proves (\ref{est wn}).\\

$\bullet$ Now we prove  \eqref{est der wn}.
As in \eqref{eq:d_n} together with \eqref{FT of wn} we have that
\begin{equation}
\begin{split}
&{\widehat{{\mc D}_n  w_n}}(\xi)=-{n}\big(1-e^{\tfrac{2i\pi \xi }{n}}\big)\Big\{{\widehat {\widetilde{\mc E}_n v}}_n (\xi) -\Big(\tfrac{1-k}{2}\Big){\widehat {\widetilde{\mc F}_n v}}_n (\xi) \Big\} J_n \big (\tfrac{\xi}{n}\big),
\end{split}
\end{equation}
where  $J_n$ is given in \eqref{eq:int_J}. Now, by using \eqref{ft_of_E_n_v} and \eqref{ft_of_F_n_v} we get
\begin{equation}
\begin{split}
{\widehat{{\mc D}_n w_n}}(\xi)&=-n^{3/2}\big( 1- e^{ \tfrac{2i \pi \xi}{n}} \big)  K_n  \big (\tfrac{\xi}{n}\big)I_n  \big (\tfrac{\xi}{n}\big)J_n(\tfrac{\xi}{n})\;   {\widehat f}_n (\xi)\\
+&4n^{3/2}\Big(\tfrac{1-k}{2}\Big) \sin^2(\tfrac{\pi \xi}{n}) K_n  \big (\tfrac{\xi}{n}\big) \widetilde {I_n}(\tfrac{\xi}{n}) J_n(\tfrac{\xi}{n})\,  \widehat{f}_n({\xi})\\
-&4n^{3/2}\Big(\tfrac{1-k}{2}\Big) \sin^2(\tfrac{\pi \xi}{n})\widetilde  K_n  \big (\tfrac{\xi}{n}\big){I_n}(\tfrac{\xi}{n})J_n(\tfrac{\xi}{n})\ ,\widehat{f}_n({\xi})\\
+&4n^{3/2}\Big(\tfrac{1-k}{2}\Big)^2 \sin^2(\tfrac{\pi \xi}{n})  \big( 1- e^{ \tfrac{-2i \pi \xi}{n}} \big)   \widetilde  K_n  \big (\tfrac{\xi}{n}\big){\widetilde I_n}(\tfrac{\xi}{n})J_n(\tfrac{\xi}{n}) \, \widehat{f}_n({\xi}).
\end{split}
\end{equation}
Finally, by 
the Plancherel-Parseval's relation we have that
\begin{equation}
\label{eq:tronconeuse}
\begin{split}
\|n^{a-b -3/2} \,  &\mc D_nw_n\|^2_{2,n}\\\leq&C  n^{2(a-b)+1} \int_{-1/2}^{1/2} |\sin (\pi y)|^2 | I_n (y)|^2| K_n (y)|^2  | J_n (y)|^2 |\widehat f_n (ny)|^2 dy\\
+&C n^{2(a-b)+1} \int_{-1/2}^{1/2} |\sin \big (\pi y\big )|^4 \big|\widetilde{I_n} (y) \big|^2| K_n (y)|^2  \big| J_n (y) \big|^2 \big| {\widehat f}_n (ny) \big|^2   dy\\
+&C n^{2(a-b)+1} \int_{-1/2}^{1/2} |\sin \big (\pi y\big )|^4 \big|{I_n} (y) \big|^2| \widetilde K_n (y)|^2  \big| J_n (y) \big|^2 \big| {\widehat f}_n (ny) \big|^2   dy\\
+&C n^{2(a-b)+1} \int_{-1/2}^{1/2} |\sin \big (\pi y\big )|^6 \big|\widetilde{I_n} (y) \big|^2| \widetilde K_n (y)|^2  \big| J_n (y) \big|^2 \big| {\widehat f}_n (ny) \big|^2   dy.
\end{split}
\end{equation}
By using Lemma \ref{lem:j_nint}, Lemma \ref{lem:in18}, Lemma \ref{lem:K_nint} and Lemma \ref{lem:sfp}, choosing a $p$ sufficiently large, we can bound the first term in on the RHS of (\ref{eq:tronconeuse}) by a constant times
\begin{equation*}
\label{eq:tournedos_new}
\begin{split}
 n^{2a -4b +1} \int_0^{1/2} \tfrac{y^5}{1+ (ny)^p} \tfrac{1}{(|y| +n^{-b})^3} \, dy = n^{2a -2 -4b} \int_{0}^{\infty} \tfrac{z^5}{1+z^p} \tfrac{1}{(z + n^{1-b})^3}\, dz. 
\end{split}
\end{equation*}
The previous integral is bounded from above by 
$$n^{2a -5 -b} \int_{0}^{\infty} \tfrac{z^5}{1+z^p}\, dz$$
which goes to $0$ since $a=\inf (3/2 -3/2b, 2) \le 2$ and $b \ge 0$.
%
%
The remaining terms at the right hand side of (\ref{eq:tronconeuse}) are proved to go to $0$ similarly by using Lemma \ref{lem:j_nint}, Lemma \ref{lem:in18} and Lemma \ref{lem:sfp}. This completes the proof of \eqref{est der wn}.\\

$\bullet$ Finally we prove \eqref{est tilde d wn}. 
We start by computing the  Fourier transform of $ {\widetilde{\mc E}_n w}_n$ and ${\widetilde{\mc F}_n w}_n $.  As
in \eqref{exp_imp_1} we have that
\begin{equation}\label{exp_imp_7}
\begin{split}
{\widehat {\widetilde{\mc E}_n w}}_n (\xi)  &=  \int_{-n/2}^{n/2}  {\widehat {w}}_n (\xi-\ell,\ell)\big\{ e^{- \tfrac{2i\pi \ell}{n}} -1 \big\} \, d\ell.
\end{split}
\end{equation}
By \eqref{FT of wn} last expression equals to
\begin{equation}\label{ft_of_E_n_w}
\begin{split}
&{\widehat {\widetilde{\mc E}_n w}}_n (\xi)  = - \cfrac{1}{n} \,\int_{-n/2}^{n/2}  \cfrac{e^{ \tfrac{2i \pi (\xi-\ell)}{n}} +e^{\tfrac{2i \pi \ell}{n}} }{  \Lambda\big( \tfrac{\xi-\ell}{n}, \tfrac{\ell}{n}\big) -i\gamma n^{-b}\, \Omega\big( \tfrac{\xi-\ell}{n}, \tfrac{\ell}{n}\big) } \; {\widehat {\widetilde{\mc E}_n v}}_n (\xi)\big\{ e^{- \tfrac{2i\pi \ell}{n}} -1 \big\} \, d\ell\\
&\quad \quad+ \Big(\tfrac{1-\kappa}{2}\Big) \, \cfrac{1}{n} \,\int_{-n/2}^{n/2}  \cfrac{e^{ \tfrac{2i \pi (\xi-\ell)}{n}} +e^{\tfrac{2i \pi \ell}{n}} }{  \Lambda\big( \tfrac{\xi-\ell}{n}, \tfrac{\ell}{n}\big) -i\gamma n^{-b}\, \Omega\big( \tfrac{\xi-\ell}{n}, \tfrac{\ell}{n}\big) }\;  {\widehat {\widetilde{\mc F}_n v}}_n (\xi)\big\{ e^{- \tfrac{2i\pi \ell}{n}} -1 \big\} \, d\ell\\
&=-\sqrt n \,  (K_n(\tfrac{\xi}{n}))^2\Big\{ I_n(\tfrac{\xi}{n})-\Big(\tfrac{1-\kappa}{2}\Big)\big(1-e^{-\frac{2\pi i \xi}{n}}\big){\widetilde I_n} (\tfrac{\xi}{n})\Big\} \; \widehat{f}_n(\xi)\\
&+\Big(\tfrac{1-\kappa}{2}\Big)\sqrt n \,  K_n(\tfrac{\xi}{n}))\widetilde K_n(\tfrac{\xi}{n})) (e^{-\tfrac{2\pi i \xi}{n}}-1)\Big\{ I_n(\tfrac{\xi}{n})-\Big(\tfrac{1-\kappa}{2}\Big)\big(1-e^{-\frac{2\pi i \xi}{n}}\big){\widetilde I_n} (\tfrac{\xi}{n})\Big\} \; \widehat{f}_n(\xi).
\end{split}
\end{equation}
Now, as in  \eqref{exp_imp_2} we have that
\begin{equation}\label{exp_imp_8}
\begin{split}
{\widehat {\widetilde{\mc F}_n w}}_n (\xi)  &=  \big( e^{ -\tfrac{2i\pi \xi }{n}}-1\big)
 \int_{-n/2}^{n/2}  {\widehat {w}}_n (\xi-\ell,\ell)\, d\ell.
\end{split}
\end{equation}
From \eqref{FT of wn} last expression is equal to 
\begin{equation}
\label{ft_of_F_n_w}
\begin{split}
&{\widehat {\widetilde{\mc F}_n w}}_n (\xi)  =    -\cfrac{1}{n} \, \Big( e^{ -\tfrac{2i\pi \xi }{n}}-1\Big)\int_{-n/2}^{n/2} \, \cfrac{e^{ \tfrac{2i \pi (\xi-\ell)}{n}} +e^{\tfrac{2i \pi \ell}{n}} }{  \Lambda\big( \tfrac{\xi-\ell}{n}, \tfrac{\ell}{n}\big) -i\gamma n^{-b}\, \Omega\big( \tfrac{\xi-\ell}{n}, \tfrac{\ell}{n}\big) }\, {\widehat {\widetilde{\mc E}_n v}}_n (\xi)  \, d\ell\\
&\quad \quad+ \cfrac{1}{n}\Big(\tfrac{1-\kappa}{2}\Big) \Big( e^{ -\tfrac{2i\pi \xi }{n}}-1\Big)\int_{-n/2}^{n/2}  \, \cfrac{e^{ \tfrac{2i \pi (\xi-\ell)}{n}} +e^{\tfrac{2i \pi \ell}{n}} }{  \Lambda\big( \tfrac{\xi-\ell}{n}, \tfrac{\ell}{n}\big) -i\gamma n^{-b}\, \Omega\big( \tfrac{\xi-\ell}{n}, \tfrac{\ell}{n}\big) } \, {\widehat {\widetilde{\mc F}_n v}}_n (\xi) d\ell\\
&= -\sqrt n \; \big( e^{ -\tfrac{2i\pi \xi }{n}}-1\big)\; {K}_n (\tfrac{\xi}{n}\big) \tilde{K}_n (\tfrac{\xi}{n})\Big\{I_n(\tfrac{\xi}{n})-\Big(\tfrac{1-k}{2}\Big)\Big(1-e^{-\frac{2\pi i \xi}{n}}\Big){\widetilde I_n} (\tfrac{\xi}{n})\Big\} \, \widehat{f}_n(\xi)\\
&+\sqrt n \Big(\tfrac{1-\kappa}{2}\Big) \; \big( e^{ -\tfrac{2i\pi \xi }{n}}-1\big)^2\;(\tilde{K}_n (\tfrac{\xi}{n}))^2\; \Big\{I_n(\tfrac{\xi}{n})-\Big(\tfrac{1-k}{2}\Big)\Big(1-e^{-\frac{2\pi i \xi}{n}}\Big){\widetilde I_n} (\tfrac{\xi}{n})\Big\} \, \widehat{f}_n(\xi).
\end{split}
\end{equation}
Now, we note that, in order to prove that
\begin{equation*}
 \|n^{a-2} \,  \tilde{\mc D_n}w_n\|^2_{2,n}
\end{equation*}
vanishes as $n\to\infty$, it is enough to show that
\begin{equation*}
 \|n^{a-1/2} \,  {\widetilde{\mc E}_n w}_n\|^2_{2,n}\quad \textrm{and}\quad 
 \|n^{a-1/2} \, {\widetilde{\mc F}_n w}_n\|^2_{2,n}
\end{equation*}
vanish as $n\to\infty$. We start with the term on the left hand side of last expression. 
From Plancherel-Parseval's relation, \eqref{ft_of_E_n_w} and the inequality $(x_1+\ldots x_k)^2 \le k [x_1^2 +\ldots +x_k^2]$, we see that we have to estimate four terms which are all of same order. One of them is 
\begin{equation*}
n^{2a} \int_{-n/2}^{n/2}   \big| K_n \big( \tfrac{\xi}{n} \big) \big|^4|I_n(\tfrac{\xi}{n})|^2 \, \big| {\widehat f}_n (\xi) \big|^2 \, d\xi.
\end{equation*}
By a change of variables last term is equal to 
\begin{equation*}
n^{2a+1} \int_{-1/2}^{1/2} \big| {\widehat f}_n (ny) \big|^2  \big| K_n \big(y \big) \big|^4|I_n(y)|^2 \, dy.
\end{equation*}
By using Lemma \ref{lem:in18}, Lemma \ref{lem:K_nint} and Lemma \ref{lem:sfp} and doing again a change of variables we bound the last term from above by
\begin{equation}
\begin{split}
 & C \; n^{2a-2b-2} \;  \int_0^{\infty} \tfrac{z^5}{1+z^p} \, \tfrac{1}{(z + n^{1-b})^3} \, dz \\
 & \le n^{b -1} \; \int_{0}^{\infty} \tfrac{z^5}{1+ z^p} \, dz
 \end{split}
\end{equation}
since $a \le 2$. This integral goes to $0$ as $n$ goes to infinity as long as $b<1$.
For the remaining integrals a similar computation can be done and the proof follows.  

\end{proof}

\section{Asymptotics of few integrals}

In this section, we compute or estimate several integrals. Some quantities are going to appear many times, therefore for the sake of clarity we introduce some notations. For any $y \in \big[-\frac{1}{2}, \frac{1}{2}\big]$ we denote by $w:=w(y)$ the complex number $w=e^{2i\pi y}$. We denote by $\mc C$ the unit circle positively oriented, and $z:=e^{2i\pi x}$ is the dummy variable  used in the complex integrals. With these notations we have
\begin{equation}
\begin{split}
&\Lambda ( y - x, x ) = 4 - z ( w^{-1} + 1 ) - z^{- 1} ( w + 1 ), \\ 
&i \Omega( y - x, x ) = z ( 1 - w^{ - 1} ) + z^{ - 1} ( w - 1 ).
\end{split}
\end{equation}
Hereafter, for any complex number $z$, we denote by $\sqrt z$ its principal square root, with positive real part. Precisely, if $z=re^{i\varphi}$, with $r \ge 0$ and $\varphi \in (-\pi,\pi]$, then the principal square root of $z$ is $\sqrt z=\sqrt r e^{i\varphi/2}.$
We introduce the degree two complex polynomial:
\begin{equation} 
P_w(z):=z^2-\frac{4}{(1+\bar w)+\gamma n^{-b} (1-\bar w)} z + w=(z-z_-)(z-z_+), 
\label{eq:pz}
\end{equation}
where $|z_-|<1$ and $|z_+|>1$. The important identities are \[
z_-z_+ =w, \qquad
z_-+z_+=\frac{4}{(1+\bar w)+\gamma n^{-b} (1-\bar w)}.\]
Finally, we denote 
\begin{align*}
a_n(w):&= (1+\bar w)+\gamma n^{-b} (1-\bar w)\\
\delta_n(w)  :&= 4- w \big[ (1+ {\bar w}) + \gamma n^{-b} (1-{\bar w}) \big]^2,
\end{align*}
so that the discriminant of $P_w$ is $4 \delta_n(w) / a_n^2 (w)$ and
\[
z_+ = \frac{2 +\sqrt{\delta_n(w)}}{a_n(w)},\qquad  z_-=\frac{2-\sqrt{\delta_n (w)}}{a_n(w)}.
\]

\begin{lemma}
\label{lem:G0345}
We have that
\begin{equation}
\begin{split}
&G_n (y) = \cfrac{1}{n^{b/2}}\;  {\sqrt{\cfrac{\gamma}{2}}}\;  |\sin(\pi y)|^{3/2} \; e^{i {\rm{sgn}} (y) \tfrac{\pi}{4}} + {\mc O} ( \sin^2 (\pi y) ) \quad \text{if} \quad b<1, \\
& G_n (y) = {\mc O} \Big( n^{-b} |\sin (\pi y)|\Big)\quad \text{if} \quad b\ge1.
\end{split}
\end{equation}
\end{lemma}
\begin{proof}
We compute the function $G_n$ by using the residue theorem. We have that
\begin{equation}
G_n (y) = \cfrac{(1-{\bar w})^2}{8i\pi} \cfrac{\gamma n^{-b}}{a_n (w)} \oint_{\mc C} f_{w} (z) dz
\end{equation}
with
\begin{equation}
 f_{w} (z) =\cfrac{(z^2+w)^2 }{z^2 P_w (z) } =\cfrac{(z^2+w)^2}{z^2 (z-z_-) (z-z_+)}.
\end{equation}
Therefore we get that
\begin{equation}
\begin{split}
G_n (y) &= \cfrac{(1-{\bar w})^2}{4} \cfrac{\gamma n^{-b}}{a_n (w)} \Big[ {\rm{Res}} (f_w,0) + {\rm{Res}} (f_{w},z_-) \Big]\\
&=\cfrac{(1-{\bar w})^2}{4} \cfrac{\gamma n^{-b}}{a_n (w)} \Big[ z_-+z_+ + \cfrac{(z_{-}^2 +w)^2}{z_{-}^2 (z_- -z_+)} \Big]\\
&=\cfrac{(1-{\bar w})^2}{4} \cfrac{\gamma n^{-b}}{a_n (w)} \Big[ \cfrac{4}{a_n (w)} -\cfrac{8}{a_n(w) \sqrt{\delta_n (w)}}  \Big]\\
&=\gamma n^{-b} \cfrac{(1-{\bar w})^2}{a_n^2 (w)} \Big[ 1-\cfrac{2}{\sqrt{\delta_n (w)}} \Big]\\
&=\cfrac{\gamma}{4n^b} \cfrac{(1- {\bar w})^2 w}{\tfrac{w}{4} a_n^2 (w)}\, \left[ 1-\cfrac{1}{\sqrt{1-\tfrac{w}{4} a_n^2 (w)}} \right]\\
&=\cfrac{\gamma }{4n^b} (1- {\bar w})^2 w g\big( \tfrac{w}{4} a_n^2 (w)\big) - \cfrac{\gamma}{4n^b} \cfrac{ (1- {\bar w})^2 w }{\sqrt{1-\tfrac{w}{4} a_n^2 (w)}}
\end{split}
\end{equation}
where $g: u \in \CC \backslash\{1\} \to u^{-1} (1- (1-u)^{-1/2})+ (1-u)^{-1/2}$. We have that $g(u)=\tfrac{1}{1+ \sqrt{1-u}}$. Since $\sqrt{1-u}$ has a positive real part, we deduce that the function $g$ is uniformly bounded. Therefore, there exists a universal constant $C>0$ such that
\begin{equation*}
\left| G_n (y) + \cfrac{\gamma}{4n^b} \cfrac{ (1- {\bar w})^2 w }{\sqrt{1-\tfrac{w}{4} a_n^2 (w)}}\right| \le Cn^{-b} \sin^2 (\pi y).
\end{equation*}
Let us now observe that  
\begin{equation}
\label{eq:tournevis1}
1-\cfrac{w}{4} a_n^2 (w) = \Big( 1+\cfrac{\gamma^2}{n^{2b}} \Big)\sin^2 (\pi y) - i\cfrac{\gamma}{n^{b}} \sin (2\pi y)
\end{equation}
and that
\begin{equation}
{\rm{Arg}} \left( 1-\cfrac{w}{4} a_n^2 (w) \right) =-{\rm{sgn}} (y) \cfrac{\pi}{2} + \arctan \left( \cfrac{ \gamma^{-1} n^b (1+ \gamma^2 n^{-2b})}{2}\tan (\pi y)\right).
\end{equation}
Since $\cos^2(\pi y)= 1- \sin^2(\pi y)$, we have that
\begin{equation}
\label{eq:tournevis2}
\begin{split}
\left| 1-\cfrac{w}{4} a_n^2 (w) \right|^2 & = \sin^2 (\pi y) \left[ (1+\gamma^2n^{-2b})^2 \sin^2 (\pi y) +4 \gamma^2 n^{-2b} \cos^2 (\pi y)  \right]\\
&=\sin^2 (\pi y) \; \left\{ \Big( 1- \tfrac{\gamma^2}{n^{2b}} \Big)\textcolor{black}{^2} \sin^2 (\pi y) + \tfrac{4 \gamma^2}{n^{2b}} \right\}.
\end{split}
\end{equation}
It follows that
\begin{equation}
\label{eq:regia}
\begin{split}
G_n (y) &= \cfrac{\gamma}{n^b} \cfrac {|\sin (\pi y)|^{3/2}}{\Big[ (1+\gamma^2 n^{-2b})^2 \sin^2 (\pi y) + 4 \gamma^2 n^{-2b} \cos^2 (\pi y) \Big]^{1/4}} \; e^{i\varphi_n (y)} \\
&+ {\mc O} (  n^{-b} \sin^2 (\pi y) )
\end{split}
\end{equation}
with 
\begin{equation}
\varphi_n (y) = {\rm{sgn}} (y) \cfrac{\pi}{4} - \cfrac{1}{2} \arctan \Big( \frac{\gamma^{-1} n^b (1+\gamma^2 n^{-2b})}{2} \tan(\pi y) \Big).
\end{equation}
Assume first that $b<1$. Then, by (\ref{eq:tournevis2}), we have that
\begin{equation}
\begin{split}
 & \cfrac{\gamma}{n^b} \cfrac {|\sin (\pi y)|^{3/2}}{\Big[ (1+\gamma^2 n^{-2b})^2 \sin^2 (\pi y) + 4 \gamma^2 n^{-2b} \cos^2 (\pi y) \Big]^{1/4}} \\
 &= \cfrac{\gamma}{n^{b/2} \sqrt{2\gamma}} \cfrac{|\sin (\pi y)|^{3/2}}{\Big[ 1+ \tfrac{n^{2b}}{4\gamma^2} (1-\gamma^2 n^{-2b})^2 \, \sin^2 (\pi y)  \Big]^{1/4}}\\
 &=\cfrac{\sqrt \gamma}{n^{b/2} \sqrt{2}} |\sin(\pi y)|^{3/2} + {\varepsilon_n} (y).  
 \end{split}
\end{equation}
We claim that $\varepsilon_n (y)= {\mc O} \left( \sin^2 (\pi y) \right)$. To prove it we distinguish two cases
\begin{itemize}
\item $|\sin(\pi y)| \ge n^{-b}$ then since $|(1+t)^{-1/4} -1| \le 2$ for $t>0$, we have
$$|\varepsilon_n (y)| \le \cfrac{\sqrt 2 \gamma |\sin(\pi y)|^{3/2}}{n^{b/2} } \le C \sin^2 (\pi y).$$  
 \item $|\sin(\pi y)| \le n^{-b}$ then since $|(1+t)^{-1/4} -1| \le t$ for $t>0$, we have
 \begin{equation*}
 \begin{split}
 |\varepsilon_n (y)| &\le C\cfrac{|\sin(\pi y)|^{3/2}}{n^{b/2}} \; n^{2b} \sin^2 (\pi y)\\
 &=C n^{3b/2}|\sin(\pi y)|^{3/2} \; \sin^2 (\pi y) \le C \sin^2 (\pi y)
 \end{split}
 \end{equation*}
\end{itemize}
and the claim is proved. Therefore, we have that for any $y\in [-1/2, 1/2]$, 
\begin{equation}
\begin{split}
G_n (y)= \cfrac{\sqrt \gamma}{n^{b/2} \sqrt{2}} |\sin(\pi y)|^{3/2} \; e^{i \varphi_n (y)} + {\mc O} ( \sin^2 (\pi y) ).
\end{split}
\end{equation}
Observe also that 
\begin{equation}
\begin{split}
\big|e^{i \varphi_n (y)} -  e^{i {\rm{sgn}} (y) \tfrac{\pi}{4}} \big| & \le   \Big| \exp\big\{-\tfrac{i}{2} \arctan \Big(\frac{\gamma^{-1} n^b (1+\gamma^2 n^{-2b})}{2}\Big) \tan(\pi y) )\big\}-1 \Big|\\
&\le \tfrac{1}{2} \Big| \arctan (\gamma^{-1} n^b (1+\gamma^2 n^{-2b}) \tan(\pi y) ) \Big| \\
&\le {\mc O} \Big (\tfrac{\pi}{2} \wedge n^b |\tan(\pi y)|\Big).
\end{split}
\end{equation}

It follows that
\begin{equation}
\begin{split}
G_n (y)& = \cfrac{\sqrt \gamma}{n^{b/2} \sqrt{2}} |\sin(\pi y)|^{3/2} \; e^{i {\rm{sgn}} (y) \tfrac{\pi}{4}}\\
 &+ {\mc O} ( \sin^2 (\pi y) ) + {\mc O} \Big (n^{-b/2} |\sin(\pi y)|^{3/2} \wedge n^{b/2} |\sin(\pi y)|^{5/2} \Big).
\end{split}
\end{equation}
By considering the cases $|\sin (\pi y)| \le n^{-b}$ and $|\sin(\pi y)| > n^{-b}$ we see that 
$$ {\mc O} \Big (n^{-b/2} |\sin(\pi y)|^{3/2} \wedge n^{b/2} |\sin(\pi y)|^{5/2} \Big)= {\mc O} (\sin^2 (\pi y))$$
and this proves the first claim.

For the second item, assume that $b\ge1$ and start with the expression (\ref{eq:regia}) and we observe that for some constant $C>0$,
\begin{equation}
\begin{split}
 & \cfrac{\gamma}{n^b} \cfrac {|\sin (\pi y)|^{3/2}}{\Big[ (1+\gamma^2 n^{-2b})^2 \sin^2 (\pi y) + 4 \gamma^2 n^{-2b} \cos^2 (\pi y) \Big]^{1/4}} \\
 & \le \cfrac{C}{n^b} \cfrac{|\sin(\pi y)|^{3/2}}{|\sin (\pi y)|^{1/2}} = \cfrac{C}{n^b} |\sin(\pi y)|.
 \end{split}
\end{equation}

\end{proof}

\begin{lemma}
\label{lem:in18}
Let $b\ge 0$. We have that for any $y \in [-1/2, 1/2]$
$$| I_n (y) | \le C \;  n^{-b} \; \cfrac{|\sin (\pi y)|^{3/2}}{\sqrt{n^{-b} + |\sin (\pi y)|}} $$
and 
$$| {\widetilde I}_n (y) | \le C \; n^{-b} \; \cfrac{|\sin (\pi y)|^{1/2}}{\sqrt{n^{-b} + |\sin (\pi y)|}}.$$

\end{lemma}

\begin{proof}
We have that
\begin{equation}
I_n (y)= -\cfrac{\gamma}{n^b} \cfrac{(1-\bar w)}{a_n (w)}\, \cfrac{1}{2i \pi} \oint_{\mc C} g_w (z)   \, dz
\end{equation}
with
\begin{equation}
\label{eq:gwin}
g_{w} (z)=\cfrac{(z-1) (z^2 +w)}{z^2 P_w(z)}=\cfrac{(z-1) (z^2 +w)}{z^2 (z-z_-)(z-z_+)}.
\end{equation}
It follows that
\begin{equation*}
\begin{split}
I_n (y)&= -\cfrac{\gamma}{n^b} \cfrac{(1-\bar w)}{a_n (w)} \; \left\{ {\rm{Res}}(g_w,0)  +{\rm{Res}}(g_w,z_-)\right\}\\
&=-\cfrac{\gamma}{n^b} \cfrac{(1-\bar w)}{a_n (w)} \left\{ 1- \cfrac{z_- + z_+}{w} +\Big( \cfrac{1}{z_{-}} -1 \Big) \cfrac{z_- + z_+}{z_+ - z_-}  \right\}\\
&= -\cfrac{\gamma}{n^b} \cfrac{1-\bar w}{a_n (w)} \; \left\{1- \cfrac{2}{w\tfrac{a_n (w)}{2}} +\cfrac{1}{\sqrt{1- \tfrac{w}{4}a_n^2 (w)}} \;  \cfrac{\tfrac{a_n (w)}{2} -1 + \sqrt{1- \tfrac{w}{4}a_n^2 (w)}}{1-\sqrt{1- \tfrac{w}{4}a_n^2 (w)} } \right\}\\
&:= -\cfrac{\gamma (1- \bar w) }{2 n^b} {\mc K}_w \Big(\tfrac{ a_{n} (w)}{2}\Big) 
\end{split}
\end{equation*}
with 
\begin{equation}
{\mc K}_w (u) =\cfrac{1}{u} \left\{1- \cfrac{2}{w u} +\cfrac{1}{\sqrt{1- w u^2}} \;  \cfrac{u -1 + \sqrt{1- w u^2}}{1-\sqrt{1- w u^2} } \right\}, \quad |u| <1.
\end{equation}
We observe first that uniformly in $w$ we have by (\ref{eq:tournevis1})
\begin{equation}
\label{eq:marteau1}
\cfrac{a_n (w)}{2} =\cfrac{1+ {\bar w}}{2} + {\mc O} \big( \tfrac{|1-w|}{n^b} \big)
\end{equation}
and by (\ref{eq:tournevis2}) that
\begin{equation}
\label{eq:marteau2}
c \, |w-1| \, \left( |w-1| + n^{-b}\right) \le \Big|1- w \cfrac{a_n^2 (w)}{4} \Big| \le C \,  |w-1| \, \left( |w-1| + n^{-b}\right).
\end{equation}
It follows that if $w$ is not close to $\pm 1$, say $|w \pm 1| \ge \ve>0$, then $ {\mc K}_w \Big(\tfrac{ a_{n} (w)}{2}\Big)$ can be uniformly bounded by a constant $C_\ve$ independently of $n$. 

If $w$ is close to $-1$ then $a_n(w)$ is close to $0$. Performing a Taylor expansion of ${\mc K}_w$ around $u=0$, we obtain that uniformly in $w$,
\begin{equation*}
{\mc K}_w (u) = {\mc O} (1).
\end{equation*} 

It remains thus only to consider the case where $w$ is close to $1$, say $|w-1| \le {\ve}$, which implies that $| a_n (w)/2 -1| \le \ve$ for $n$ sufficiently large. We rewrite, for $|w-1| + |u-1| \le \ve$, 
\begin{equation}
\begin{split}
|{\mc K}_w (u)| &=\left|\cfrac{1}{u} \left\{1- \cfrac{2}{w u} + \cfrac{1}{1-\sqrt{1- w u^2}} +\cfrac{u-1}{\sqrt{1- w u^2}} \;  \cfrac{1}{1-\sqrt{1- w u^2} } \right\} \right|\\
&\le C_{\ve} \left[ 1 + \cfrac{|u-1|}{\sqrt{|1- w u^2|}} \right].
\end{split}
\end{equation}
We use now (\ref{eq:marteau1}) and (\ref{eq:marteau2}) to obtain
\begin{equation}
\Big| {\mc K}_w \Big(\tfrac{ a_{n} (w)}{2}\Big) \Big| \le C_{\ve} \left[ 1+ \cfrac{ \sqrt{|w-1|} }{\sqrt{ |w-1| +n^{-b} }} \right] \le C_{\ve} \cfrac{ \sqrt{|w-1|} }{\sqrt{ |w-1| +n^{-b} }}.
\end{equation} 

The conclusion of the first item follows.

Similarly, we have that
\begin{equation}
{\widetilde{I_n}} (y)= -\cfrac{\gamma}{n^b} \cfrac{1-\bar w}{a_n (w)}\;\cfrac{1}{2i \pi} \oint_{\mc C} h_w (z)   \, dz
\end{equation}
with
\begin{equation}
\label{eq:gwin2}
h_{w} (z)=\cfrac{(z^2 +w)}{z P_w (z) }=\cfrac{(z^2 +w)}{z (z-z_-)(z-z_+)}.
\end{equation}
It follows that
\begin{equation*}
\begin{split}
{\widetilde I}_n (y)&= -\cfrac{\gamma}{n^b} \cfrac{(1-\bar w)}{a_n (w)} \; \left\{ {\rm{Res}}(h_w,0)  +{\rm{Res}}(h_w,z_-) \right\}\\
&= \cfrac{\gamma (1-\bar w) }{2n^b} \, \widetilde{\mc K}_w \Big( \tfrac{a_n (w)}{2}\Big)
\end{split}
\end{equation*}
where
$$\widetilde{\mc K}_w (u)= \cfrac{1}{u} \; \left\{ 1- \cfrac{1}{\sqrt{1- w u^2}}\right\}, \quad |u| <1.$$ 
Let $\ve>0$ small be fixed. If $|w\pm 1| \ge \ve$ then by (\ref{eq:marteau1}) we deduce that for $n$ sufficiently large (uniformly in $w$ in this domain) 
$$\left| \widetilde{\mc K}_w \Big( \tfrac{a_n (w)}{2}\Big) \right| \le C_\ve.$$
If $|w +1| \le \ve$, then for $n$ sufficiently large, uniformly in $w$ in this domain, $|a_n (w)| \le \ve$. And we have that for $|u| \le \ve$, uniformly in $w$, $|{\mc K}_{w} (u)| \le C u$.Therefore we deduce that in this case
 $$\left|\widetilde{\mc K}_w \Big( \tfrac{a_n (w)}{2}\Big) \right| \le C_\ve.$$
If $|w -1| \le \ve$, then for $n$ sufficiently large, uniformly in $w$ in this domain, $|a_n (w) /2 -1| \ge \ve/2$ and therefore, by (\ref{eq:marteau2}) we have
 $$\Big| \widetilde{\mc K}_w \Big( \tfrac{a_n (w)}{2}\Big) \Big| \le  \cfrac{C_{\ve} }{\sqrt{|w-1|^2 + |w-1| n^{-b} }}.$$
The conclusion of the second item follows.
\end{proof}

\begin{lemma}
\label{lem:j_nint}
Let $b\ge 0$. We have that for any $y \in [-1/2, 1/2]$
$$| J_n (y) | \le C \; \cfrac{|\sin (\pi y)|^{-1/2}}{\sqrt{n^{-b} + |\sin (\pi y)|}} $$

\end{lemma}

\begin{proof}
We have that
\begin{equation}
J_n (y)= - \cfrac{1}{(1+\bar w)+\gamma n^{-b} (1-\bar w)}\, \cfrac{1}{2i \pi} \oint_{\mc C} k_w (z)   \, dz
\end{equation}
with
\begin{equation}
k_{w} (z)=\cfrac{(z^2 +w)}{z^2 {P}_w (z)}
\end{equation}
It follows that
\begin{equation*}
\begin{split}
J_n (y)&= -  \cfrac{1}{(1+\bar w)+\gamma n^{-b} (1-\bar w)} \; \left\{ {\rm{Res}}(k_w,0)  +{\rm{Res}}(k_w,z_-) \right\}\\
&=  \cfrac{1}{(1+\bar w)+\gamma n^{-b} (1-\bar w)} \; \left\{ \cfrac{{z}_+ +{z}_-}{w} -\cfrac{1}{{z}_-} \; \cfrac{{z}_+ + {z}_-}{{z}_+ -{z}_-} \right\}\\
&= {\mc H}_w \Big(\tfrac{{a_n} (w)}{2} \Big)
\end{split}
\end{equation*}
with
$${\mc H}_w (u) = \cfrac{1}{u} \left\{ \cfrac{2}{w u} - \cfrac{1}{\sqrt{1-w u^2}} \, \cfrac{u}{1- \sqrt{1- wu^2} }\right\}, \quad |u| <1.$$
Since uniformly in $w$ we have ${\mc H}_w (u) = {\mc O} (1)$ for $u \to 0$, we have only to study the behavior of ${\mc H}_w \Big(\tfrac{{a_n} (w)}{2} \Big)$ for $w$ close to $1$, which implies $a_n (w) /2$ close to $1$ by recalling (\ref{eq:marteau1}). Therefore, for say $|w-1| \le \ve$ with $\ve>0$ small, we have
\begin{equation}
\Big| {\mc H}_w \Big(\tfrac{{a_n} (w)}{2} \Big) \Big| \le C_\ve \left[ 1+ \cfrac{1}{\sqrt{\Big|1- w \tfrac{a_n^2 (w)}{4} \Big|}}\right].
\end{equation}
Taking into account (\ref{eq:marteau2}) we get the claim.
\end{proof}

%

\begin{lemma}
\label{lem:K_nint}
Let $b\ge 0$. We have that for any $y \in [-1/2, 1/2]$
$$| K_n (y) | \le C \; \cfrac{|\sin (\pi y)|^{1/2}}{\sqrt{n^{-b} + |\sin (\pi y)|}} $$
and
$$| {\widetilde K}_n (y) | \le C\; \cfrac{|\sin (\pi y)|^{-1/2}}{\sqrt{n^{-b} + |\sin (\pi y)|}}.$$
\end{lemma}

\begin{proof}
We have that
\begin{equation*}
K_n (y) =  \cfrac{1}{a_n (w)} \cfrac{1}{2i \pi}\oint_{\mc C}  g_{w} (z) dz 
\end{equation*}
where
$$g_w (z) = \cfrac{(z-1)(z^2 +w)}{z^2 P_{w} (z)}$$
has been introduced in (\ref{eq:gwin}) during the proof of Lemma \ref{lem:in18}. We have therefore that
\begin{equation}
K_n (y)= -\cfrac{n^b}{\gamma} \cfrac{1}{1- {\bar w}}\; I_n (y). 
\end{equation}
The estimate on $K_n$ follows from Lemma \ref{lem:in18}.

Similarly we have that
\begin{equation*}
{\widetilde K_n} (y) =  -\cfrac{1}{a_n (w)} \cfrac{1}{2i \pi}\oint_{\mc C}  h_{w} (z) dz 
\end{equation*}
where
$$h_w (z) = \cfrac{z^2 +w}{z P_{w} (z)}$$
has been introduced in (\ref{eq:gwin2}) during the proof of Lemma \ref{lem:in18}. We have therefore that
\begin{equation}
{\widetilde K_n} (y)=\cfrac{n^b}{\gamma} \cfrac{1}{1- {\bar w}}\; {\widetilde I}_n (y). 
\end{equation}
The estimate on $\widetilde K_n$ follows from Lemma \ref{lem:in18}.

\end{proof}

\bibliographystyle{plain}


\end{document}